\documentclass[11pt,reqno]{amsart}

\usepackage{times}
\usepackage[T1]{fontenc}
\usepackage{mathrsfs}
\usepackage{latexsym}
\usepackage[dvips]{graphics}
\usepackage{graphicx}
\usepackage{fancyhdr,amssymb}
\usepackage{epsfig}
\usepackage{amsmath,amsfonts,amsthm,amssymb,amscd}
\input amssym.def
\input amssym.tex
\usepackage{color}
\usepackage{hyperref}
\usepackage{url}

\newcommand{\burl}[1]{\textcolor{blue}{\url{#1}}}

\usepackage[margin=1in]{geometry}
\usepackage{mathtools}

\newcommand\be{\begin{equation}}
\newcommand\ee{\end{equation}}
\newcommand\bea{\begin{eqnarray}}
\newcommand\eea{\end{eqnarray}}

\theoremstyle{plain}
\numberwithin{equation}{section}
\newtheorem{thm}{Theorem}[section]

\newtheorem{cor}[thm]{Corollary}
\newtheorem{lem}[thm]{Lemma}
\newtheorem{prop}[thm]{Proposition}
\newtheorem{exa}[thm]{Example}
\newtheorem{defi}[thm]{Definition}
\newtheorem{defn}[thm]{Definition}
\newtheorem{remark}[thm]{Remark}

\begin{document}

\fancyhead{}
\renewcommand{\headrulewidth}{0pt}
\fancyfoot{}
\fancyfoot[LE,RO]{\medskip \thepage}
\fancyfoot[LO]{\medskip MONTH YEAR}
\fancyfoot[RE]{\medskip VOLUME , NUMBER }

\setcounter{page}{1}

\title{A Generalization of Fibonacci Far-Difference Representations and Gaussian Behavior}

\author[Demontigny]{Philippe Demontigny}
\email{\textcolor{blue}{\href{mailto:Philippe.P.Demontigny@williams.edu}{Philippe.P.Demontigny@williams.edu}}
}
\address{Department of Mathematics \& Statistics, Williams College, Williamstown, MA 01267}

\author[Do]{Thao Do}
\email{\textcolor{blue}{\href{mailto:thao.do@stonybrook.edu}{thao.do@stonybrook.edu}}}
\address{Mathematics Department, Stony Brook University, Stony Brook, NY, 11794}

\author[Kulkarni]{Archit Kulkarni}
\email{\textcolor{blue}{\href{mailto:auk@andrew.cmu.edu}{auk@andrew.cmu.edu}}}
\address{Department of Mathematical Sciences, Carnegie Mellon University, Pittsburgh, PA 15213}

\author[Miller]{Steven J. Miller}
\email{\textcolor{blue}{\href{mailto:sjm1@williams.edu}{sjm1@williams.edu}},  \textcolor{blue}{\href{Steven.Miller.MC.96@aya.yale.edu}{Steven.Miller.MC.96@aya.yale.edu}}}
\address{Department of Mathematics and Statistics, Williams College, Williamstown, MA 01267}

\author[Varma]{Umang Varma}
\email{\textcolor{blue}{\href{mailto:Umang.Varma10@kzoo.edu}{Umang.Varma10@kzoo.edu}}}
\address{Department of Mathematics \& Computer Science, Kalamazoo College, Kalamazoo, MI, 49006}

\thanks{This research was conducted as part of the 2013 SMALL REU program at Williams College and was partially supported funded by NSF grant DMS0850577 and Williams College; the fourth named author was also partially supported by NSF grant DMS1265673. We would like to thank our  colleagues from the Williams College 2013 SMALL REU program for helpful discussions, especially Taylor Corcoran, Joseph R. Iafrate, David Moon, Jaclyn Porfilio, Jirapat Samranvedhya and Jake Wellens, and the referee for many helpful comments.}

\begin{abstract} A natural generalization of base $B$ expansions is Zeckendorf's Theorem, which states that every integer can be uniquely written as a sum of non-consecutive Fibonacci numbers $\{F_n\}$, with $F_{n+1} = F_n + F_{n-1}$ and $F_1=1, F_2=2$. If instead we allow the coefficients of the Fibonacci numbers in the decomposition to be zero or $\pm 1$, the resulting expression is known as the far-difference representation. Alpert proved that a far-difference representation exists and is unique under certain restraints that generalize non-consecutiveness, specifically that two adjacent summands of the same sign must be at least 4 indices apart and those of opposite signs must be at least 3 indices apart.

In this paper we prove that a far-difference representation can be created using sets of Skipponacci numbers, which are generated by recurrence relations of the form $S^{(k)}_{n+1} = S^{(k)}_{n} + S^{(k)}_{n-k}$ for $k \ge 0$. Every integer can be written uniquely as a sum of the $\pm S^{(k)}_n $'s such that every two terms of the same sign differ in index by at least $2k+2$, and every two terms of opposite signs differ in index by at least $k+2$. Let $I_n = (R_k(n-1), R_k(n)]$ with $R_k(\ell) = \sum_{0 < \ell-b(2k+2) \le \ell} S^{(k)}_{\ell-b(2k+2)}$. We prove that the number of positive and negative terms in given Skipponacci decompositions for $m \in I_n$ converges to a Gaussian as $n\to\infty$, with a computable correlation coefficient. We next explore the distribution of gaps between summands, and show that for any $k$ the probability of finding a gap of length $j \ge 2k+2$ decays geometrically, with decay ratio equal to the largest root of the given $k$-Skipponacci recurrence. We conclude by finding sequences that have an $(s,d)$ far-difference representation (see Definition \ref{def:sdfardiffrep}) for any positive integers $s,d$.
\end{abstract}



\maketitle

\tableofcontents

\section{Introduction}

In this paper we explore signed decompositions of integers by various sequences. After briefly reviewing the literature, we state our results about uniqueness of decomposition, number of summands, and gaps between summands. In the course of our analysis we find a new way to interpret an earlier result about far-difference representations, which leads to a new characterization of the Fibonacci numbers.

\subsection{Background}

Zeckendorf \cite{Ze} discovered an interesting property of the Fibonacci numbers $\{F_n\}$; he proved that every positive integer can be written uniquely as a sum of non-consecutive Fibonacci numbers\footnote{If we were to use the standard definition of $F_0 = 0$, $F_1 = 1$ then we would lose uniqueness.}, where $F_{n+2} = F_{n+1} + F_n$ and $F_1 = 1, F_2 = 2$. It turns out this is an alternative characterization of the Fibonacci numbers; they are the unique increasing sequence of positive integers such that any positive number can be written uniquely as a sum of non-consecutive terms.

Zeckendorf's theorem inspired many questions about the number of summands in these and other decompositions. Lekkerkerker \cite{Lek} proved that the average number of summands in the decomposition of an integer in $[F_n, F_{n+1})$ is $\frac{n}{\varphi^2+1} + O(1)$, where $\varphi = \frac{1+\sqrt{5}}{2}$ is the golden mean (which is the largest root of the characteristic polynomial associated with the Fibonacci recurrence). More is true; as $n\to\infty$, the distribution of the number of summands of $m \in [F_n, F_{n+1})$ converges to a Gaussian. This means that as $n\to\infty$ the fraction of $m \in [F_n, F_{n+1})$ such that the number of summands in $m$'s Zeckendorf decomposition is in $[\mu_n - a\sigma_n, \mu_n + b\sigma_n]$ converges to $\frac1{\sqrt{2\pi}} \int_a^b e^{-t^2/2}dt$, where $\mu_n = \frac{n}{\varphi^2+1} + O(1)$ is the mean number of summands for $m \in [F_n, F_{n+1})$ and $\sigma_n^2 = \frac{\varphi}{5(\varphi+2)}n-\frac{2}{25}$ is the variance (see \cite{KKMW} for the calculation of the variance). \emph{Henceforth in this paper whenever we say the distribution of the number of summand converges to a Gaussian, we mean in the above sense.} There are many proofs of this result; we follow the combinatorial approach used in \cite{KKMW}, which proved these results by converting the question of how many numbers have exactly $k$ summands to a combinatorial one.

These results hold for other recurrences as well. Most of the work in the field has focused on Positive Linear Recurrence Relations (PLRS), which are recurrence relations of the form $G_{n+1} = c_1G_n + \cdots + c_L G_{n+1-L}$ for non-negative integers $L,c_1,c_2,\dots,c_L$ with $L,c_1,$ and $c_n > 0$ (these are called $G$-ary digital expansions in \cite{St}). There is an extensive literature for this subject; see \cite{Al,BCCSW,Day,GT,Ha,Ho,Ke,Len,MW1,MW2} for results on uniqueness of decomposition and \cite{DG,FGNPT,GTNP,KKMW,Lek,LT,MW1,St} for Gaussian behavior.

Much less is known about signed decompositions, where we allow negative summands in our decompositions. This opens up a number of possibilities, as in this case we can overshoot the value we are trying to reach in a given decomposition, and then subtract terms to reach the desired positive integer. We formally define this idea below.

\begin{defn}[Far-difference representation]
A \emph{far-difference representation} of a positive integer $x$ by a sequence $\{a_n\}$ is a signed sum of terms from the sequence which equals $x$.
\end{defn}

The Fibonacci case was first considered by Alpert \cite{Al}, who proved the following analogue of Zeckendorf's theorem. Note that the restrictions on the gaps between adjacent indices in the decomposition is a generalization of the non-adjacency condition in the Zeckendorf decomposition.

\begin{thm} \label{thm:alpert}
Every $x \in \mathbb{Z}$ has a unique Fibonacci far-difference representation such that every two terms of the same sign differ in index by at least 4 and every two terms of opposite sign differ in index by at least 3.
\end{thm}

For example, 2014 can be decomposed as follows:
\be
2014 \ = \ 2584 - 610 + 55 - 13 - 2 \ = \ F_{17} - F_{14} + F_9 - F_6 - F_2.
\ee
Alpert's proof uses induction on a partition of the integers, and the method generalizes easily to other recurrences which we consider in this paper.

Given that there is a unique decomposition, it is natural to inquire if generalizations of Lekkerkerker's Theorem and Gaussian behavior hold as well. Miller and Wang \cite{MW1} proved that they do. We first set some notation, and then describe their results (our choice of notation is motivated by our generalizations in the next subsection).

First, let $R_4(n)$ denote the following summation

\begin{equation} \label{R4(n)}
R_4(n) \ := \
\begin{cases}
\sum_{0 < n-4i \le n} F_{n-4i} \ = \  F_n + F_{n-4} + F_{n-8} + \cdots & \text{ if } n > 0 \\
0 & \text{ otherwise.}
\end{cases}
\end{equation}

Using this notation, we state the motivating theorem from Miller-Wang.

\begin{thm}[Miller-Wang] \label{thm:MW1 result}
Let $\mathcal{K}_n$ and $\mathcal{L}_n$ be the corresponding random variables denoting the number of positive summands and the number of negative summands in the far-difference representation (using the signed Fibonacci numbers) for integers in $(R_4(n-1), R_4(n)]$. As $n$ tends to infinity, $\mathbb{E}[\mathcal{K}_n] = \frac{1}{10}n + \frac{371-113\sqrt{5}}{40} + o(1)$, and is $\frac{1+\sqrt{5}}{4} = \frac{\varphi}{2}$ greater than $\mathbb{E}[\mathcal{L}_n]$. The variance of both is $\frac{15 + 21\sqrt{5}}{1000}n + O(1)$. The standardized joint density of $\mathcal{K}_n$ and $\mathcal{L}_n$ converges to a bivariate Gaussian with negative correlation $\frac{10\sqrt{5}-121}{179} = -\frac{21-2\varphi}{29+2\varphi} \approx -0.551$, and $\mathcal{K}_n + \mathcal{L}_n$ and $\mathcal{K}_n - \mathcal{L}_n$ converge to independent random variables.
\end{thm}

Their proof used generating functions to show that the moments of the distribution of summands converge to those of a Gaussian. The main idea is to show that the conditions which imply Gaussianity for positive-term decompositions also hold for the Fibonacci far-difference representation. One of our main goals in this paper is to extend these arguments further to the more general signed decompositions. In the course of doing so, we find a simpler way to handle the resulting algebra.

We then consider an interesting question about the summands in a decomposition, namely \emph{how are the lengths of index gaps between adjacent summands distributed in a given integer decomposition?} Equivalently, how long must we wait after choosing a term from a sequence before the next term is chosen in a particular decomposition? In \cite{BBGILMT}, the authors solve this question for the Fibonacci far-difference representation, as well as other PLRS, provided that all the coefficients are positive. Note this restriction therefore excludes the $k$-Skipponaccis for $k \ge 2$.

\begin{thm}[\cite{BBGILMT}]\label{thm:skipgaps}
As $n \to \infty$, the probability $P(j)$ of a gap of length $j$ in a far-difference decomposition of integers in $(R_4(n-1), R_4(n)]$ converges to geometric decay for $j \ge 4$, with decay constant equal to the golden mean $\varphi$. Specifically, if $a_1 = \varphi / \sqrt{5}$ (which is the coefficient of the largest root of the recurrence polynomial in Binet's Formula\footnote{As our Fibonacci sequence is shifted by one index from the standard representation, for us Binet's Formula reads $F_n = \frac{\varphi}{\sqrt{5}} \varphi^n - \frac{1-\varphi}{\sqrt{5}} (1-\varphi)^n$. For any linear recurrence whose characteristic polynomial is of degree $d$ with $d$ distinct roots, the $n$\textsuperscript{{\rm th}} term is a linear combination of the $n$\textsuperscript{{\rm th}} powers of the $d$ roots; we always let $a_1$ denote the coefficient of the largest root.} expansion for $F_n$), then $P(j) = 0$ if $j \le 2$ and
\begin{equation} \label{thm:FibonacciGaps}
P(j) \ = \
\begin{cases}
\frac{10a_1\varphi}{\varphi^4-1}\varphi^{-j} & \text{ if } j \ge 4 \\
\frac{5a_1}{\varphi^2(\varphi^4-1)} & \text{ if } j = 3.
\end{cases}
\end{equation}
\end{thm}

\subsection{New Results}

In this paper, we study far-difference relations related to certain generalizations of the Fibonacci numbers, called the $k$-Skipponacci numbers.

\begin{defi}[$k$-Skipponacci Numbers] For any non-negative integer $k$, the $k$-Skipponaccis are the sequence of integers defined by $S^{(k)}_{n+1} = S^{(k)}_n + S^{(k)}_{n-k}$ for some $k \in \mathbb{N}$. We index the $k$-Skipponaccis such that the first few terms are $S^{(k)}_1 = 1$, $S^{(k)}_2 = 2$, ..., $S^{(k)}_{k+1} = k+1$, and $S^{(k)}_n = 0$ for all $n \le 0$. \end{defi}

Some common $k$-Skipponacci sequences are the 0-Skipponaccis (which are powers of 2, and lead to binary decompositions) and the 1-Skipponaccis (the Fibonaccis). Our first result is that a generalized Zeckendorf theorem holds for far-difference representations arising from the $k$-Skipponaccis.

\begin{thm}\label{Thm:Far-Diff}
Every $x \in \mathbb{Z}$ has a unique far-difference representation for the $k$-Skipponaccis such that every two terms of the same sign are at least $2k+2$ apart in index and every two terms of opposite sign are at least $k+2$ apart in index.
\end{thm}

Before stating our results on Gaussianity, we first need to set some new notation, which generalizes the summation in \eqref{R4(n)}. \begin{equation} \label{Rn}
R_k(n) \ := \
\begin{cases}
\sum_{0 < n-b(2k+2) \le n} S^{(k)}_{n-b(2k+2)} \ = \  S^{(k)}_n + S^{(k)}_{n-2k-2} + S^{(k)}_{n-4k-4} +  \cdots & \text{ if } n > 0
\\
0 & \text{ otherwise,}
\end{cases}
\end{equation}

\begin{thm} \label{thm:Gaussianity[MW]} Fix a positive integer $k$. Let $\mathcal{K}_n$ and $\mathcal{L}_n$ be the corresponding random variables denoting the number of positive and the number of negative summands in the far-difference representation for integers in $(R_k(n-1),R_k(n)]$ from the $k$-Skipponaccis. As $n\to\infty$, expected values of $\mathcal{K}_n$ and $\mathcal{L}_n$ both grow linearly with $n$ and differ by a constant, as do their variances. The standardized joint density of $\mathcal{K}_n$ and $\mathcal{L}_n$ converges to a bivariate Gaussian with a computable correlation. More generally, for any non-negative numbers $a, b$ not both equal to $0$, the random variable $X_n=a\mathcal{K}_n+b\mathcal{L}_n$ converges to a normal distribution as $n\to\infty$.
\end{thm}

\noindent This theorem is an analogue to Theorem \ref{thm:MW1 result} of \cite{MW1} for the case of Fibonacci numbers. Their proof, which is stated in Section 6 of \cite{MW1}, relies heavily on Section 5 of the same paper where the authors proved Gaussianity for a large subset of sequences whose generating function satisfies some specific constraints. In this paper we state a sufficient condition for Gaussianity in the following theorem, which we prove in \S\ref{sec:gaussianity}. We show that it applies in our case, yielding a significantly simpler proof of Gaussianity than the one in \cite{MW1}.

\begin{thm}\label{thm_generalGaussian}
Let $\kappa$ be a fixed positive integer. For each $n$, let a discrete random variable $X_n$ in $I_n=\{0,1,\dots,n\}$ have
\be {\rm Prob}(X_n=j)\ = \
\begin{cases}
  \rho_{j;n}/\sum_{j=1}^n  \rho_{j;n} & \text{{\rm if} } j\in I_n \\
 0 &\text{{\rm otherwise}}
\end{cases}
  \ee
for some positive real numbers $ \rho_{1;n},\dots, \rho_{n;n}$. Let $g_n(x) := \sum_j  \rho_{j;n}x^j$ be the generating function of $X_n$. If $g_n$ has form $g_n(x)\ = \ \sum_{i=1}^\kappa q_i(x)\alpha_i^n(x)$ where
\begin{itemize}
\item[(i)] for each $i\in\{1,\dots,\kappa\}$, $q_i,\alpha_i:\mathbb{R}\to\mathbb{R}$ are three times differentiable functions  which do not depend on $n$;
\item[(ii)] there exists some small positive $\epsilon$ and some positive constant $\lambda<1$ such that for all $x\in I_\epsilon=[1-\epsilon,1+\epsilon]$, $|\alpha_1(x)|>1$ and $\frac{|\alpha_i(x)|}{|\alpha_1(x)|}<\lambda<1$ for all $i=2,\dots,\kappa$;
\item[(iii)] $\alpha_1'(1)\neq 0$ and $\frac{d}{dx}\left[\frac{x\alpha_1'(x)}{\alpha_1(x)}\right] \left|_{x=1}\neq 0\right.$;
\end{itemize}
then
\begin{itemize}
\item[(a)]  The mean $\mu_n$ and variance $\sigma_n^2$ of $X_n$ both grow linearly with $n$. Specifically,
\begin{equation}
\mu_n\ = \ A n+B+o(1)
\end{equation}
\begin{equation}
\sigma_n^2\ = \ C \cdot n+ D+o(1)
\end{equation}
where \begin{equation}A\ = \ \frac{\alpha_1'(1)}{\alpha_1(1)}, \ \ \ \  B\  = \ \frac{q_1'(1)}{q_1(1)}
\end{equation}
\begin{equation}
C\ = \ \left(\frac{x\alpha_1'(x)}{\alpha_1(x)}\right)\Bigg|_{x=1}\ = \ \frac{\alpha_1(1)[\alpha_1'(1)+\alpha_1''(1)]-\alpha_1'(1)^2}{\alpha_1(1)^2}
\end{equation}
\begin{equation}
D\ = \ \left(\frac{xq_1'(x)}{q_1(x)}\right)\Bigg|_{x=1} \ = \ \frac{q_1(1)[q_1'(1)+q_1''(1)]-q_1'(1)^2}{q_1(1)^2}.
\end{equation}
\item[(b)] As $n\to\infty$, $X_n$ converges in distribution to a normal distribution.
\end{itemize}
\end{thm}

Next we generalize previous work on gaps between summands. This result makes use of a standard result, the Generalized Binet's Formula; see \cite{BBGILMT} for a proof for a large family of recurrence relations which includes the $k$-Skipponaccis. We restate the result here for the specific case of the $k$-Skipponaccis.

\begin{lem} \label{Binet-Skipponacci}
Let $\lambda_1,\dots,\lambda_k$ be the roots of the characteristic polynomial for the $k$-Skipponaccis. Then $\lambda_1 > |\lambda_2| \ge \cdots \ge |\lambda_k|$, $\lambda_1 > 1$, and there exists a constant $a_1$ such that
\begin{equation}
S^{(k)}_n \ = \  a_1\lambda_1^n + O(n^{\max(0,k-2)}\lambda_2^n).
\end{equation}
\end{lem}

\begin{thm} \label{thm:gapresult} Consider the $k$-Skipponacci numbers $\{S^{(k)}_n\}$. For each $n$, let $P_n(j)$ be the probability that the size of a gap between adjacent terms in the far-difference decomposition of a number $m \in (R_k(n-1),R_k(n)]$ is $j$. Let $\lambda_1$ denote the largest root of the recurrence relation for the $k$-Skipponacci numbers, and let $a_1$ be the coefficient of $\lambda_1$ in the Generalized Binet's formula expansion for $S^{(k)}_n$. As $n\to\infty$, $P_n(j)$ converges to geometric decay for $j \ge 2k+2$, with computable limiting values for other $j$. Specifically, we have $ \lim_{n\to\infty}P_n(j) = P(j) = 0$ for $j \le k+1$, and
\begin{equation}
P(j) \ = \  \begin{cases}
\frac{a_1\lambda_1^{-3k-2}}{A_{1,1} \left(1-\lambda_1^{-2k-2}\right)^2 (\lambda_1-1)}\lambda_1^{-j} & \text{if }\; k+2 \le j < 2k+2 \\
\frac{a_1\lambda_1^{-2k-2}}{A_{1,1} \left(1-\lambda_1^{-2k-2}\right)^2 (\lambda_1-1)}\lambda_1^{-j} & \text{if }\; j \ge 2k+2.
\end{cases}
\end{equation}
where $A_{1,1}$ is a constant defined in \eqref{E[K+L]}.
\end{thm}

Our final results explore a complete characterization of sequences that exhibit far-difference representations. That is, we study integer decompositions on a sequence of terms in which same sign summands are $s$ apart in index and opposite sign summands are $d$ apart in index. We call such representations \emph{(s,d) far-difference representations}, which we formally define below.

\begin{defn}[$(s,d)$ far-difference representation]\label{def:sdfardiffrep} A sequence $\{a_n\}$ has an \emph{$(s,d)$ far-difference representation} if every integer can be written uniquely as sum of terms $\pm a_n$ in which every two terms of the same sign are at least $s$ apart in index and every two terms of opposite sign are at least $d$ apart in index.
\end{defn}

Thus the Fibonaccis lead to a $(4,3)$ far-difference representation. More generally, the $k$-Skipponaccis lead to a $(2k+2,k+2)$ one. We can consider the reverse problem; if we are given a pair of positive integers $(s,d)$, is there a sequence such that each number has a unique $(s,d)$ far-difference representation? The following theorem shows that the answer is yes, and gives a construction for the sequence.

\begin{thm}\label{farDiffRec} Fix positive integers $s$ and $d$, and define a sequence $\{a_n\}_{n=1}^{\infty}$ by
\begin{itemize}
\item[i.] For $n=1,2,\dots,\min(s,d)$, let $a_n=n$.
\item[ii.] For $\min(s,d)< n\leq \max(s,d)$, let
\be a_n \ = \  \left\{
  \begin{array}{l l}
    a_{n-1}+a_{n-s} & \quad \text{{\rm if}\ $s<d$}\\
    a_{n-1}+a_{n-d}+1 & \quad \text{{\rm if}\ $d\leq s$.}
  \end{array} \right.\ee

\item[iii.] For $n> \max(s,d)$, let $a_n=a_{n-1}+a_{n-s}+a_{n-d}$.
\end{itemize}
Then the sequence $\{a_n\}$ has an unique $(s,d)$ far-difference representation.
\end{thm}

In particular, as the Fibonaccis give rise to a $(4,3)$ far-difference representation, we should have $F_n = F_{n-1} + F_{n-4} + F_{n-3}$. We see this is true by repeatedly applying the standard Fibonacci recurrence: \bea F_n  \ = \  F_{n-1} + F_{n-2} \ = \ F_{n-1} + \left(F_{n-3} + F_{n-4}\right) \ = \ F_{n-1} + F_{n-4} + F_{n-3}. \eea

To prove our results we generalize the techniques from \cite{Al, BBGILMT, MW1} to our families. In \S\ref{sec:fardiffrepskip} we prove that for any $k$-Skipponacci recurrence relation, a unique far-difference representation exists for all positive integers. In \S\ref{sec:gaussianity} we prove that the number of summands in any far-difference representation approaches a Gaussian, and then we study the distribution of gaps between summands in \S\ref{sec:distrgaps}. We end in \S\ref{sec:genfardiffseq} by exploring generalized $(s,d)$ far-difference representations.

\section{Far-difference representation of $k$-Skipponaccis}\label{sec:fardiffrepskip}

Recall the $k$-Skipponaccis satisfy the recurrence $S^{(k)}_{n+1} = S^{(k)}_n + S^{(k)}_{n-k}$ with $S^{(k)}_i = i$ for $1 \le i \le k+1$. Some common $k$-Skipponacci sequences are the 0-Skipponaccis (the binary sequence) and the 1-Skipponaccis (the Fibonaccis). We prove that every integer has a unique far-difference representation arising from the $k$-Skipponaccis. The proof is similar to Alpert's proof for the Fibonacci numbers.

We break the analysis into integers in intervals $(R_k(n-1), R_k(n)]$, with $R_k(n)$ as in \eqref{Rn}. We need the following fact.

\begin{lem} \label{Lem:R+R=S-1} Let $\{S^{(k)}_n\}$ be the $k$-Skipponacci sequence. Then
\begin{equation} \label{lemma1}
S^{(k)}_{n} - R_k(n-k-2) - R_k(n-1)=1.
\end{equation}
\end{lem}

The proof of follows by a simple induction argument, which for completeness we give in Appendix \ref{sec:proofsfromsecfardiffreplemmas}.

\begin{proof}[Proof of Theorem \ref{Thm:Far-Diff}] It suffices to consider the decomposition of positive integers, as negative integers follow similarly. Note the number 0 is represented by the decomposition with no summands.

We claim that the positive integers are the disjoint union over all closed intervals of the form $[S^{(k)}_n - R_k(n-k-2), R_k(n)]$. To prove this, it suffices to show that $S^{(k)}_{n} - R_k(n-k-2) = R_k(n-1) + 1$ which follows immediately from Lemma \ref{Lem:R+R=S-1}.


Assume a positive integer $x$ has a $k-$Skipponacci far-differenced representation in which $S^{(k)}_n$ is the leading term, (i.e., the term of largest index). It is easy to see that because of our rule, the largest number can be decomposed with the leading term $S^{(k))}_n$ is $ S^{(k)}_n+S^{(k)}_{n-2k-2}+S^{(k)}_{n-4k-4}+\cdots=R_k(n)$ and the smallest one is $S^{(k)}_n-S^{(k)}_{n-k-2}-S^{(k)}_{n-3k-4}-\cdots=S^{(k)}_n-R_k(n-k-2)$, hence $S^{(k)}_n-R_k(n-k-2)\leq x\leq R_k(n)$. Since we proved that $\{[S^{(k)}_n - R_k(n-k-2), R_k(n)]\}_{n=1}^\infty$ is a disjoint cover of all positive integers, for any integer $x\in \mathbb{Z}^+$, there is a unique $n$ such that $S^{(k)}_n - R_k(n-k-2) \le x \le S^{(k)}_n$. Further, if $x$ has a $k$-Skipponacci far-difference representation, then $S^{(k)}_n$ must be its leading term.

Therefore if a decomposition of such an $x$ exists it must begin with $S^{(k)}_n$. We are left with proving a decomposition exists and that it is unique. We proceed by induction.

For the base case, let $n=0$. Notice that the only value for $x$ on the interval $0 \le x \le R_k(0)$ is $x=0$, and the $k$-Skipponacci far-difference representation of $x$ is empty for any $k$. Assume that every integer $x$ satisfying $0 \le x \le R_k(n-1)$ has a unique far-difference representation. We now consider $x$ such that $R_k(n-1) < x \le R_k(n)$. From our partition of the integers, $x$ satisfies $S^{(k)}_n - R_k(n-k-2) \le x \le R_k(n)$. There are two cases.

\begin{itemize}

\item[(1)] $S^{(k)}_n - R_k(n-k-2) \le x \le S^{(k)}_n$. \\
Note that for this case, it is equivalent to say $0 \le S^{(k)}_n - x \le R_k(n-k-2)$. It then follows from the inductive step that $S^{(k)}_n - x$ has a unique $k$-Skipponacci far-difference representation with $S^{(k)}_{n-k-2}$ as the upper bound for the main term.

\item[(2)] $S^{(k)}_n \le x \le R_k(n)$. \\
For this case, we can once again subtract $S^{(k)}_n$ from both sides of the inequality to get $0 \le x-S^{(k)}_n \le R_k(n-2k-2)$. It then follows from the inductive step that $x-S^{(k)}_n$ has a unique far-difference representation with main term at most $S^{(k)}_{n-2k-2}$.

\end{itemize}

In either case, we can generate a unique $k$-Skipponacci far-difference representation for $x$ by adding $S^{(k)}_n$ to the representation for $x - S^{(k)}_n$ (which, from the definition of $R_k(m)$, in both cases has the index of its largest summand sufficiently far away from $n$ to qualify as a far-difference representation. \end{proof}

\section{Gaussian Behavior}\label{sec:gaussianity}

In this section we follow method in Section 6 of \cite{MW1} to prove Gaussianity for the number of summands. We first find the generating function for the problem, and then analyze that function to complete the proof.

\subsection{Derivation of the Generating Function}\label{sec:derivgenfns}

Let $p_{n,m,\ell}$ be the number of integers in $(R_k(n)$, $R_k(n+1)]$ with exactly $m$ positive summands and exactly $\ell$ negative summands in their far-difference decomposition via the $k$-Skipponaccis (as $k$ is fixed, for notational convenience we suppress $k$ in the definition of $p_{n,m,\ell}$). When $n \le 0$ we let $p_{n,m,\ell}$ be 0.  We first derive a recurrence relation for $p_{n,m,\ell}$ by a combinatorial approach, from which the generating function immediately follows.
\begin{lem} Notation as above, for $n > 1$ we have
\begin{equation} \label{prec1}
p_{n,m,\ell}\ = \ p_{n-1,m,\ell}+ p_{n-(2k+2),m-1,\ell} + p_{n-(k+2),\ell,m-1}.
\end{equation}
\end{lem}

\begin{proof} First note that $p_{n,m,\ell} = 0$ if $m \le 0$ or $\ell < 0 $. In \S\ref{sec:fardiffrepskip} we partitioned the integers into the intervals $[R_k(n-1)+1,R_k(n)]$, and noted that if an integer $x$ in this interval has a far-difference representation, then it must have leading term $S^{(k)}_n$, and thus $x - S^{(k)}_n \in [R_k(n-1)+1-S^{(k)}_n,R_k(n)-S^{(k)}_n]$. From Lemma \ref{Lem:R+R=S-1} we have
\bea\label{S-R_n-1-R_n-k-2=1}
S^{(k)}_n - R_k(n-1) - R_k(n-k-2)
\ = \ 1,
\eea which implies $R_k(n-1) + 1 - S^{(k)}_n = -R_k(n-k-2)$. Thus $p_{n,m,\ell}$ is the number of far-difference representations for integers in $[-R_k(n-k-2), R_k(n-2k-2)]$ with $m-1$ positive summands and $\ell$ negative summands (as we subtracted away the main term $S^{(k)}_n$).

Let $n > 2k+2$. There are two possibilities.\\

\noindent \texttt{Case 1: $(k-1,\ell) = (0,0)$.}

\noindent Since $S^{(k)}_n - R_k(n-1) - R_k(n-k-2) = 1$ by \eqref{S-R_n-1-R_n-k-2=1}, we know that $S^{(k)}_{n-1} < R_k(n-1) < S^{(k)}_n$ for all $n > 1$. This means there must be exactly one $k$-Skipponacci number on the interval $[R_k(n-1)+1,R_k(n)]$ for all $n > 1$. It follows that $p_{n,1,0} = p_{n-1,1,0} = 1$, and the recurrence in \eqref{prec1} follows since $p_{n-k-2,0,0}$ and $p_{n-2k-2,0,0}$ are both 0 for all $n > 2k+2$. \\

\noindent \texttt{Case 2: $(k-1,\ell)$ is not $(0,0)$.}

\noindent Let $N(I,m,\ell)$ be the number of far-difference representations of integers in the interval $I$ with $m$ positive summands and $\ell$ negative summands. Thus
\begin{align} \label{pnml_sum1}
p_{n,m,\ell}
\;\ = \ &\; N\left[ (0,R_k(n-2k-2)],m-1,\ell \right] + N\left[ (-R_k(n-k-2),0],m-1,\ell \right] \nonumber \\
\;\ = \ &\; N\left[ (0,R_k(n-2k-2)],m-1,\ell \right] + N\left[ (0,R_k(n-k-2)],\ell,m-1 \right] \nonumber \\
\;\ = \ &\; \sum_{i=1}^{n-2k-2} p_{i,m-1,\ell} + \sum_{i=1}^{n-k-2} p_{i,\ell,m-1}.
\end{align}

Since $n > 1$, we can replace $n$ with $n-1$ in \eqref{pnml_sum1} to get
\begin{equation} \label{pnml_sum2}
p_{n-1,m,\ell}
\;\ = \ \; \sum_{i=1}^{n-2k-3} p_{i,m-1,\ell} + \sum_{i=1}^{n-k-3} p_{i,\ell,m-1}.
\end{equation}
Subtracting \eqref{pnml_sum2} from\eqref{pnml_sum1} gives us the desired expression for $p_{n,m,\ell}$. \end{proof}

The generating function $G_k(x,y,z)$ for the far-difference representations by $k$-Skipponacci numbers is defined by \be G_k(x,y,z)\ =\ \sum p_{n,m,\ell}x^my^{\ell}z^n. \ee

\begin{thm} \label{Thm:G_k(x,y,z)} Notation as above, we have
\begin{equation} \label{genfn}
G_k(x,y,z)
\;\ = \ \; \frac{xz-xz^2+xyz^{k+3}-xyz^{2k+3}}{1-2z+z^2-(x+y)z^{2k+2}+(x+y)z^{2k+3}-xyz^{2k+4}+xyz^{4k+4}}.
\end{equation}
\end{thm}

\begin{proof} Note that the equality in \eqref{prec1} holds for all triples $(n,m,\ell)$ except for the case where $n=1$, $m=1$, and $\ell=0$ under the assumption that $p_{n,m,\ell}=0$ whenever $n\leq 0$. To prove the claimed formula for the generating function in \eqref{genfn}, however, we require a recurrence relation in which each term is of the form $p_{n-n_0,m-m_0,\ell-\ell_0}$. This can be achieved with some simple substitutions. Replacing $(n,m,\ell)$ in \eqref{prec1} with $(n-k-2,\ell,m-1)$ gives
\begin{equation} \label{prec2}
p_{n-k-2,\ell,m-1}\ = \ p_{n-(k+3),\ell,m-1}+ p_{n-(3k+4),\ell-1,m-1} + p_{n-(2k+4),m-1,\ell-1},
\end{equation} which holds for all triples except $(k+3,1,1)$. Rearranging the terms of \eqref{prec1}, we get
\begin{equation} \label{prec3}
p_{n-(k+2),\ell,m-1} \ = \  p_{n,m,\ell} - p_{n-1,m,\ell} - p_{n-(2k+2),m-1,\ell}.
\end{equation}

We replace $(n,m,\ell)$ in \eqref{prec3} with $(n-1,m,\ell)$ and $(n-2k-2,m,\ell-1)$ which yields
\begin{equation} \label{prec4}
p_{n-(k+3),l,m-1} \ = \  p_{n-1,m,l} - p_{n-2,m,l} - p_{n-(2k+3),m-1,l},
\end{equation} which only fails for the triple $(2,1,0)$, and
\begin{equation} \label{prec5}
p_{n-(3k+4),l-1,m-1} \ = \  p_{n-(2k+2),m,l-1} - p_{n-(2k+3),m,l-1} - p_{n-(4k+4),m-1,l-1},
\end{equation} which only fails for the triple $(2k+3,1,1)$. We substitute equations \eqref{prec3}, \eqref{prec4} and \eqref{prec5} into \eqref{prec1} and obtain the following expression for $p_{n,m,\ell}$:
\begin{align} \label{pnmlrec}
p_{n,m,l}
\;\ = \ &\; 2p_{n-1,m,l} - p_{n-2,m,l} + p_{n-(2k+2),m-1,l}  + p_{n-(2k+2),m,l-1} \nonumber \\
\;&\; - p_{n-(2k+3),m-1,l} - p_{n-(2k+3),m,l-1} + p_{n-(2k+4),m-1,l-1} - p_{n-(4k+4),m-1,l-1}.
\end{align}

Using this recurrence relation, we prove that the generating function in \eqref{genfn} is correct. Consider the following characteristic polynomial for the recurrence in \eqref{prec5}:
\begin{equation} \label{Pxyz}
P(x,y,z)
\ = \  1 - 2z + z^2 -(x+y)z^{2k+2} + (x+y)z^{2k+3} - xyz^{2k+4} + xyz^{4k+4}.
\end{equation}
We take the product of this polynomial with the generating function to get
\begin{align} \label{GenRec}
P(x,y,z)G_k(x,y,z)
\;\ = \ &\; \left( 1 - 2z + z^2 -(x+y)z^{2k+2} + (x+y)z^{2k+3} - xyz^{2k+4}\right. \nonumber \\
\;&\; \left. + xyz^{4k+4}\right) \cdot \sum_{n \ge 1} p_{n,m,l}x^my^lz^n \nonumber \\
\;\ = \ &\; x^my^lz^n \cdot \sum_{n \ge 1} p_{n,m,l} - 2p_{n-1,m,l} + p_{n-2,m,l} - p_{n-(2k+2),m-1,l} \nonumber \\
\;&\; - p_{n-(2k+2),m,l-1} + p_{n-(2k+3),m-1,l} + p_{n-(2k+3),m,l-1} \nonumber \\
\;&\; - p_{n-(2k+4),m-1,l-1} + p_{n-(4k+4),m-1,l-1}.
\end{align}

Notice that the equality from \eqref{prec5} appears within the summation, and this quantity is zero whenever the equality holds. We have shown  that the only cases where a triple does not satisfy the equality is when $(n,m,\ell)$ is given by $(1,1,0)$, $(2,1,0)$, $(k+3,1,1)$ or $(2k+3,1,1)$. Since \eqref{pnmlrec} is a combination of \eqref{prec3}, \eqref{prec4}, \eqref{prec2} and \eqref{prec5}, where these triples fail, it follows that they will also not satisfy the equality in \eqref{pnmlrec}. Thus within the summation in \eqref{GenRec}  we are left with a non-zero coefficient for $x^my^{\ell}z^n$. We collect these terms and are left with the following:
\begin{equation}
P(x,y,z)G_k(x,y,z) \ = \  xz - xz^2 + xyz^{k+3} - xyz^{2k+3}.
\end{equation}
Rearranging these terms and substituting in our value for $P(x,y,z)$ gives us the desired equation for the generating function.
\end{proof}

Going forward, we often need the modified version of our generating function in which we factor out the term $(1-z)$ from both the numerator and the denominator:
\begin{align} \label{Genfn2}
G_k(x,y,z)
\;\ = \ &\; \frac{ xz + \frac{1-z^k}{1-z}xyz^{k+3}  }{1-z-(x+y)z^{2k+2} + \frac{1-z^{2k}}{1-z}\left(-xyz^{2k+4}\right) } \nonumber \\
\;\ = \ &\; \frac{xz + xy\sum_{j=k+3}^{2k+2}z^j}{1-z-(x + y)z^{2k+2}-xy\sum_{j=2k+4}^{4k+3}z^j}.
\end{align}
For some calculations, it is more convenient to use this form of the generating function because the terms of the denominator are of the same sign (excluding the constant term).

\subsection{Proof of Theorem \ref{thm:Gaussianity[MW]}}\label{sec:subsecgaussianity}

Now that we have the generating function, we turn to proving Gaussianity. As the calculation is long and technical, we quickly summarize the main idea. We find, for $\kappa = 4k+3$, that we can write the relevant generating function as a sum of $\kappa$ terms. Each term is a product, and there is no $n$-dependence in the product (the $n$ dependence surfaces by taking one of the terms in the product to the $n$\textsuperscript{th} power). We then mimic the proof of the Central Limit Theorem. Specifically, we show only the first of the $\kappa$ terms contributes in the limit. We then Taylor expand and use logarithms to understand its behavior. The reason everything works so smoothly is that we almost have a fixed term raised to the $n$\textsuperscript{th} power; if we had that, the Central Limit Theorem would follow immediately. All that remains is to do some book-keeping to see that the mean is of size $n$ and the standard deviation of size $\sqrt{n}$.\\

To prove Theorem \ref{thm:Gaussianity[MW]}, we first prove that for each non-negative $(a,b)\neq (0,0)$, $X_n=a\mathcal{K}_n+b\mathcal{L}_n$ converges to a normal distribution as $n$ approaches infinity.

Let $x=w^a$ and $y=w^b$, then the coefficient of $z^n$ in \eqref{genfn} is given by $\sum_{m,\ell} p_{n,m,\ell}x^my^{\ell}=\sum_{m,\ell} p_{n,m,\ell} w^{am+b\ell}$. Define
\begin{equation}
g_n(w) \ := \  \sum_{m>0,\ell\ge 0} p_{n,m,\ell}w^{am + b\ell}.
\end{equation}
Then $g_n(w)$ is the generating function of $X_n$ because for each $i\in\{1,\dots,n\}$,
\begin{equation}
P(X_n=i)\ = \ \sum_{am+b\ell =i}p_{n,m,\ell}.
\end{equation}
We want to prove $g_n(w)$ satisfies all the conditions stated in Theorem \ref{thm_generalGaussian}. The following proposition, which is proved in Appendix \ref{sec:propmainres}, is useful for that purpose.

\begin{prop} There exists $\epsilon \in (0,1)$ such that for any $w \in I_{\epsilon} = (1-\epsilon,1+\epsilon)$:\label{prop:mainres}

\begin{itemize}
\item[(a)] $A_w(z)$ has no multiple roots, where $A_w(z)$ is the denominator of \eqref{genfn}.

\item[(b)] There exists a single positive real root $e_1(w)$ such that $e_1(w) < 1$ and there exists some positive $\lambda<1$ such that $|e_1(w)|/|e_i(w)|<\lambda$ for all $i \ge 2$.

\item[(c)] Each root $e_i(w)$ is continuous, infinitely differentiable, and
\begin{equation} \label{eprime}
e_1'(w)\ = \ -\frac{(aw^{a-1}+bw^{b-1})e_1(w)^{2k+2}+(a+b)w^{a+b-1}\sum_{j=2k+4}^{4k+3}e_1(w)^j}{1+(w^a+w^b)(2k+2)e_1(w)^{2k+1}+w^{a+b}
\sum_{j=2k+4}^{4k+3}je_1(w)^{j-1}}.
\end{equation}
\end{itemize}
\end{prop}

In the next step, we use partial fraction decomposition of $G_k(x,y,z)$ (from Theorem \ref{Thm:G_k(x,y,z)}) to find a formula for $g_n(w)$. Let $A_w(z)$ be the denominator of $G_k$.  Making the substitution $(x,y) = (w^a,w^b)$, we have
\begin{align} \label{pfA_w(z)}
\frac{1}{A_w(z)}
\;\ = \ &\; \frac{1}{w^{a+b}} \sum_{i\ = \ 1}^{4k+3} \frac{1}{(z-e_i(w))\prod_{j \neq i}(e_j(w) - e_i(w))} \nonumber \\
\;\ = \ &\; \frac{1}{w^{a+b}}\sum_{i=1}^{4k+3} \frac{1}{(1-\frac{z}{e_i(w)})} \cdot \frac{1}{e_i(w)\prod_{j \neq i}(e_j(w) - e_i(w))}.
\end{align}

Using the fact that $\frac{1}{1-\frac{z}{e_i(w)}}$ represents a geometric series, we combine the numerator of our generating function with our expression for the denominator in \eqref{pfA_w(z)} to get
\begin{align}
g_n(w)
\;\ = \ &\; \sum_{i=1}^{4k+3} \frac{1}{w^b e_i^n(w)\prod_{j \neq i}(e_j(w) - e_i(w))} -\sum_{i=1}^{4k+3} \frac{1}{w^b e_i^{n-1}(w)\prod_{j \neq i}(e_j(w) - e_i(w))} \nonumber \\
\;&\; + \sum_{i=1}^{4k+3} \frac{1}{e_i^{n-k-2}(w)\prod_{j \neq i}(e_j(w) - e_i(w))} - \sum_{i=1}^{4k+3} \frac{1}{e_i^{n-2k-2}(w)\prod_{j \neq i}(e_j(w) - e_i(w))} \nonumber \\
\;\ = \ &\; \sum_{i=1}^{4k+3} \frac{w^{-b}(1 - e_i(w)) + e_i^{k+2}(w) - e_i^{2k+2}(w)}{e_i^n(w)\prod_{j \neq i}(e_j(w) - e_i(w))}.
\end{align}
Let $q_i(w)$ denote all terms of $g_n(w)$ that do not depend on $n$:
\begin{equation} \label{q(w)}
q_i(w) \ := \  \frac{w^{-b}(1 - e_i(w)) + e_i^{k+2}(w) - e_i^{2k+2}(w)}{\prod_{j \neq i}(e_j(w) - e_i(w))}.
\end{equation}
Setting $\alpha_i:\ = \ 1/e_i$, we can find $g_n(w) = \sum_{i=1}^{4k+3} q_i(w)\alpha_i^n$. We want to apply Theorem \ref{thm_generalGaussian} to $X_n$. All the notations are the same except $\kappa:=4k+3$.

Indeed, by part (c) of Proposition \ref{prop:mainres}, $e_i(w)$ are infinitely many times differentiable for any $i=1,\dots,4k+3$. Since $0$ is not a root of $A_w(z)$, for sufficiently small $\epsilon$, $e_i(w)\neq 0$ for all $w\in I_\epsilon$. Therefore $\alpha_i$ and $q_i$, as rational functions of $e_1,\dots,e_{4k+3}$, are also infinitely many times differentiable; in particular, they are three times differentiable, thus satisfy condition $(i)$ in Theorem \ref{thm_generalGaussian}. By part (b) of Proposition \ref{prop:mainres}, $|e_1(w)|<1$ and $|e_1(w)|/|e_i(w)|<\lambda<1$ for $i\geq 2$. This implies $|\alpha_1(w)|>1$ and $|\alpha_i(w)|/|\alpha_1(w)|<\lambda<1$ for $i\geq 2$, thus $g_n$ satisfies condition $(ii)$ in Theorem \ref{thm_generalGaussian}. The following lemma, whose proof is stated in Appendix \ref{sec:proof_lem_variance_grow}, verifies the last condition.

\begin{lem}\label{lem_variance_grow} Given conditions as above:

\begin{equation}\label{nonzero_mean}
\frac{\alpha_1'(1)}{\alpha_1(1)}\ = \ \frac{-e'_1(1)}{e_1(1)}\ \neq \ 0.
\end{equation}
 \begin{equation}\label{nonzero_variance}
\frac{d}{dw}\left[\frac{w\alpha_1'(w)}{\alpha_1(w)}\right] \Big|_{w=1}\ = \ -
\frac{d}{dw}\left[\frac{we_1'(w)}{e_1(w)}\right] \Big|_{w=1}\ \neq  \ 0.
\end{equation}
\end{lem}

We can now apply Theorem \ref{thm_generalGaussian} to conclude that $X_n$ converges to a Gaussian as $n$ approaches infinity. Moreover, we have formulas for the mean and variance of $X_n=a\mathcal{K}_n+b\mathcal{L}_n$ for each $(a,b)$ non-negative and not both zero. We have
\begin{equation} \label{E[K+L]}
\mathbb{E}[a\mathcal{K}_n+b\mathcal{L}_n] \ = \  A_{a,b}n + B_{a,b} + o(1),\end{equation}
where $A_{a,b}=\alpha'_1(1)/\alpha_1(1)$ and $B_{a,b}=q_1'(1)/q_1(1)$, which depend only on our choice of $a$ and $b$. Further,
\begin{equation} \label{Var[K+L]}
{\rm Var}(a\mathcal{K}_n + b\mathcal{L}_n) \ = \  C_{a,b}n + D_{a,b} + o(1),
\end{equation} where $C_{a,b}\ = \ \left(\frac{w\alpha_1'(w)}{\alpha_1(w)}\right)'\Big|_{w=1}$ and
$D_{a,b}\ = \ \left(\frac{wq_1'(w)}{q_1(w)}\right)'\Big|_{w=1}$,
which depend only on $a$ and $b$. By lemma \ref{lem_variance_grow}, $A_{a,b}$ and $C_{a,b}$ are non-zero, thus the mean and variance of $X_n$ always grows linearly with $n$.

As proved above, $X_n=a\mathcal{K}_n+b\mathcal{L}_n$ converges to a Gaussian distribution as $n\to\infty$. Let $(a,b)=(1,0)$ and $(0,1)$ we get $\mathcal{K}_n$ and $\mathcal{L}_n$ individually converge to a Gaussian. By \eqref{E[K+L]}, their means both grows linearly with $n$.
\begin{equation}
\mathbb{E}[\mathcal{K}_n]=A_{1,0}n+B_{1,0}+o(1)
\end{equation}
\begin{equation}
\mathbb{E}[\mathcal{L}_n]=A_{0,1}n+B_{0,1}+o(1)
\end{equation}
Moreover, $A_{a,b}=A_{b,a}$ because $A_{a,b}=\frac{\alpha_1'(1)}{\alpha_1(1)}=\frac{-e_1'(1)}{e_1(1)}$ where $e_1(1)$ is a constant and $e'_1(1)$ is symmetric between $a$ and $b$ as shown in \eqref{eprime}. In particular $A_{1,0}=A_{0,1}$, hence $\mathbb{E}[\mathcal{K}_n]-\mathbb{E}[\mathcal{L}_n]$ converges to a constant as $n\to\infty$. This implies the average number of positive and negative summands differ by a constant.

Equation \eqref{Var[K+L]} gives us a way to calculate variance of any joint density of $\mathcal{K}_n$ and $\mathcal{L}_n$. We can furthermore calculate the covariance and correlation of any two joint densities as a function of $e_1$ and $q_1$.

In particular, we prove that $\mathcal{K}_n+\mathcal{L}_n$ and $\mathcal{K}_n-\mathcal{L}_n$ have correlation decaying to zero with $n$. Indeed, from \eqref{Var[K+L]}:
\begin{equation}
{\rm Var}[\mathcal{K}_n]\ = \ C_{1,0}n+D_{1,0}+o(1).
\end{equation}
\begin{equation}
{\rm Var}[\mathcal{L}_n]\ = \ C_{0,1}n+D_{0,1}+o(1).
\end{equation}
\noindent Note that $C_{0,1}=C_{1,0}$ because again we have \be C_{a,b}\ = \  \left(\frac{x\alpha_1'(w)}{\alpha_1(w)}\right)'\Bigg|_{w=1}\ = \ - \left(\frac{we_1'(w)}{e_1(w)}\right)'\Big|_{w=1}\ee
where $e_1(w)$ does not depend on $a,b$ and $e'_1(w)$ is symmetric between $a,b$. Therefore,
\begin{equation}
{\rm Cov}[\mathcal{K}_n+\mathcal{L}_n,\mathcal{K}_n-\mathcal{L}_n] \ = \ \frac{{\rm Var}[2\mathcal{K}_n]+{\rm Var}[2\mathcal{L}_n]}{4} \ = \ {\rm Var}[\mathcal{K}_n]-{\rm Var}[\mathcal{L}_n]\ = \ O(1).
\end{equation}
Therefore
\begin{equation}
{\rm Corr}[\mathcal{K}_n,\mathcal{L}_n]=\frac{{\rm Cov}[\mathcal{K}_n, \mathcal{L}_n]}{\sqrt{{\rm Var}[\mathcal{K}_n]{\rm Var}[\mathcal{L}_n]}}\ = \ \frac{O(1)}{\theta(n)}\ =\ o(1)
\end{equation} (where $\theta(n)$ represents a function which is on the order of $n$). This implies $\mathcal{K}_n-\mathcal{L}_n$ and $\mathcal{K}_n,\mathcal{L}_n$ are uncorrelated as $n\to\infty$. This completes the proof of Theorem \ref{thm:Gaussianity[MW]}. \hfill $\Box$

\subsection{Proof of Theorem \ref{thm_generalGaussian}}

We now collect the pieces. The argument here is different than the one used in \cite{MW1}, and leads to a conceptually simpler proof (though we do have to wade through a good amount of algebra). The rest of this section is just mimicking the standard proof of the Central Limit Theorem, while at the same time isolating the values of the mean and variance.\\

To prove part $(a)$, we use the generating function $g_n(x)$ to calculate $\mu_n$ and $\sigma^2_n$ as follows:
\begin{equation}
\mu_n\ = \ \mathbb{E}[X_n]\ = \ \frac{\sum_{i=1}^n \rho_{i;n}\cdot i}{\sum_{i=1}^n \rho_{i;n}}\ = \ \frac{g_n'(1)}{g_n(1)}
\end{equation}

\begin{equation}
\sigma_n^2\ = \ \mathbb{E}[X_n^2]-\mu_n^2\ = \ \frac{\sum_{i=1}^n \rho_{i;n}\cdot i^2}{\sum_{i=1}^n \rho_{i;n}}-\mu_n^2 \ = \ \frac{[xg'_n(x)]'\big|_{x=1}}{g_n(1)}-\left(\frac{g_n'(1)}{g_n(1)}\right)^2.
\end{equation}
The calculations are then  straightforward:
\begin{equation}
g_n'(x)\ = \ \sum_{i=1}^\kappa [q_i(x)\alpha_i^n(x)]'\ = \ \sum_{i=1}^\kappa [q_i'(x)\alpha_i^n(x)+q_i(x)n\alpha_i^{n-1}(x)\alpha'_i(x)]
\end{equation}
\begin{align}\label{variance_formula}
[xg'_n(x)]' & \ = \ \sum_{i=1}^\kappa  \left(x[q_i'(x)\alpha_i^n(x)+q_i(x)n\alpha_i^{n-1}(x)\alpha'_i(x)]\right)'\nonumber\\
&\ = \ \sum_{i=1}^\kappa \left( q_i'(x)\alpha_i^n(x)+q_i(x)n\alpha_i^{n-1}(x)\alpha_i'(x)+\right.\nonumber\\
& \left. x\left[ q_i''(x)\alpha_i^n(x)+2q_i'(x)n\alpha_i^{n-1}(x)\alpha_i'(x)+q_in\alpha_i^{n-1}\alpha_i''(x)+q_i(x)n(n-1)\alpha_i^{n-2}(\alpha_i'(x))^2\right]\right).
\end{align}
Since $|\alpha_i(1)/\alpha_1(1)|<\lambda<1$ for each $i\geq 2$, we have
\begin{equation}
\sum_{i=2}^\kappa q_i(1)\alpha_i^n(1)\ = \ \alpha_1^n(1)\sum_{i=2}^\kappa q_i(1)\left(\frac{\alpha_i(1)}{\alpha_1(1)}\right)^n\ = \ o(\lambda^n)\alpha_1^n(1).
\end{equation}
Similarly,
\begin{equation}
\sum_{i=2}^\kappa [q_i(x)\alpha_i^n(x)]'\Big|_{x=1}\ = \ \alpha_1^n(1)\sum_{i=2}^\kappa \left[q'_i(1)+\frac{nq_i(1)\alpha'_i(1)}{\alpha'_i(1)}
\right]\left(\frac{\alpha_i(1)}{\alpha_1(1)}\right)^n\ = \ o(\lambda^n)\alpha_1^n(1)
\end{equation} and
\begin{equation}
\sum_{i=2}^\kappa \Big(x[q_i(x)\alpha_i^n(x)]'\Big)'\Big|_{x=1}\ = \ o(\lambda^n)\alpha_1^n(1).
\end{equation}
Hence
\begin{align}
\mu_n\ = \ \frac{g'_n(1)}{g_n(1)}& \ = \ \frac{[q_1'(1)\alpha_1^n(1)+q_1(1)n\alpha_i^{n-1}(1)\alpha'_1(1)]+ o(\lambda^n) \alpha_1^n(1)}{q_1(1)\alpha_1^n(1)+o(\lambda^n) \alpha_1^n(1)}\nonumber\\
&\ = \ \frac{q_1'(1)+q_1(1)n\frac{\alpha'_1(1)}{\alpha_1(1)}+o(\lambda^n)}{q_1(1)+o(\lambda^n)}
\ = \ \frac{q_1'(1)}{q_1(1)}+n\frac{\alpha_1'(1)}{\alpha_1(1)}+o(1).
\end{align}
Similarly,
\begin{align}
\sigma_n^2 & \ = \ \frac{[xg'_n(x)]'\big|_{x=1}}{g_n(1)}-\mu_n^2\nonumber\\ &\ = \ \frac{([x(q_1(x)\alpha_1(x))']'\Big|_{x=1}+o(\lambda^n)\alpha_1^n(1)}{q_1(1)\alpha_1^n(1)+o(\lambda^n)\alpha_1^n(1)}-\mu_n^2\nonumber\\
&\ = \ \frac{q_1'}{q_1}+\frac{n\alpha_1'}{\alpha_1}+\frac{q''_1}{q_1(1)}+\frac{2q'_1n\alpha_1'}{\alpha_1}+\frac{n\alpha_1''}{\alpha_1}+\frac{n(n-1)(\alpha'_1)^2}{\alpha_1^2}-\left(\frac{\alpha'_1}{\alpha_1}n+\frac{q'_1}{q_1}+o(1)\right)^2\nonumber\\
& \ = \ \frac{\alpha_1(\alpha_1'+\alpha_1'')-(\alpha_1')^2}{\alpha_1^2}\cdot n+\frac{q_1(q_1'+q_1'')-(q_1')^2}{q_1^2}+o(1).
\end{align}
Here we apply \eqref{variance_formula} and use $q_1,\alpha_1$ short for $q_1(1),\alpha_1(1)$. The last things we need are
\be \frac{\alpha_1(1)[\alpha_1'(1)+\alpha_1''(1)]-\alpha_1'(1)^2}{\alpha_1(1)^2}
\ = \ \left(\frac{x\alpha_1'(x)}{\alpha_1(x)}\right)\Bigg|_{x=1}\ee
and
\be\frac{q_1(1)[q_1'(1)+q_1''(1)]-q_1'(1)^2}{q_1(1)^2}\ = \ \left(\frac{xq_1'(x)}{q_1(x)}\right)\Bigg|_{x=1},\ee
which are simple enough to check directly. This completes the proof of part $(a)$ of Theorem \ref{thm_generalGaussian}.\\ \

To prove part $(b)$ of the theorem, we use the method of moment generating functions, showing that moment generating function of $X_n$ converges to that of a Gaussian distribution as $n\to\infty$. (We could use instead the characteristic functions, but the moment generating functions have good convergence properties here.) The moment generating function of $X_n$ is
\be M_{X_n}(t)=\mathbb{E}[e^{tX_n}]\ = \ \frac{\sum_i \rho_{i;n} e^{ti}}{\sum_i {\rho_{i;n}}}\ = \ \frac{g_n(e^t)}{g_n(1)}\ = \ \frac{\sum_{i=1}^\kappa q_i(e^t)\alpha_i^n(e^t)}{\sum_{i=1}^\kappa q_i(1)\alpha_i^n(1)}.\ee

Since $|\alpha_i(e^t)|<|\alpha_1(e^t)|$ for any $i\geq 2$, the main term of $g_n(e^t)$ is $q_1(e^t)\alpha_1(e^t)$. We thus write
\begin{align}
M_{X_n}(t) &\ = \   \frac{\sum_{i=1}^\kappa q_i(e^t)\alpha_i^n(e^t)}{\sum_{i=1}^\kappa q_i(1)\alpha_i^n(1)}
\ = \  \frac{q_1(e^t)\alpha_1^n(e^t)\left[1+\sum_{i=2}^k \frac{q_i(e^t)}{q_1(e^t)}\left(\frac{\alpha_i(e^t)}{\alpha_1(e^t)}\right)^n\right]}{q_1(1)\alpha_1^n(1)\left[1+\sum_{i=2}^\kappa \frac{q_i(1)}{q_1(1)}\left(\frac{\alpha_i(1)}{\alpha_1(1)}\right)^n\right]}\nonumber\\
&\ = \ \frac{q_1(e^t)\alpha_1^n(e^t)[1+ O(\kappa Q\lambda^n)]}{q_1(1)\alpha_1^n(1)[1+ O(\kappa Q\lambda^n)]}
\ = \ \frac{q_1(e^t)}{q_1(1)}\left(\frac{\alpha_1(e^t)}{\alpha_1(1)}\right)^n\left(1+O(\kappa Q\lambda^n)\right),
\end{align}
where $Q=\max_{i\geq 2} \sup_{t\in [-\delta,+\delta]} \frac{q_i(e^t)}{q_1(e^t)}$. As $0<\lambda<1$, $\kappa Q\lambda^n$ rapidly decays when $n$ gets large. Taking the logarithm of both sides yields
\begin{equation}
\log M_{X_t}\ = \ \log \frac{q_1(e^t)}{q_1(1)}+n\log\frac{\alpha_1(e^t)}{\alpha_1(1)}+\log\left(1+O(\kappa Q\lambda^n)\right)\ = \ \log \frac{q_1(e^t)}{q_1(1)}+n\log\frac{\alpha_1(e^t)}{\alpha_1(1)}+o(1).
\end{equation}
Let
$Y_n=\frac{X_n-\mu_n}{\sigma_n}$, then the moment generating function of $Y_n$ is
\begin{equation}
M_{Y_n}(t)\ = \ \mathbb{E}[e^{t(X_n-\mu_n)/\sigma_n}]\ = \ M_{X_n}(t/\sigma_n) e^{-t\mu_n/\sigma_n}.
\end{equation}
Therefore
\begin{equation}\label{log_M_Yn}
\log M_{Y_n}(t)\ = \ \frac{-t\mu_n}{\sigma_n}+ \log \frac{q_1(e^{t/\sigma_n})}{q_1(1)}+n\log\frac{\alpha_1(e^{t/\sigma_n})}{\alpha_1(1)}+o(1).
\end{equation}
Since $\sigma_n=\theta (\sqrt{n})$, $t/\sigma_n\to 0$ as $n\to\infty$. Hence \begin{equation}\label{log_q}
\lim_{n\to\infty} \log \frac{q_1(e^{t/\sigma_n})}{q_1(1)}\ = \ \log 1\ = \ 0.
\end{equation}
Using the Taylor expansion of degree two at 1,  we can write $\alpha_1(x)$ as
\begin{equation}
\alpha_1(x)=\alpha_1(1)+\alpha'(1)(x-1)+\frac{\alpha_1''(1)}{2} (x-1)^2+O((x-1)^3).
\end{equation}
Substituting $x=e^{t/\sigma_n}=1+\frac{t}{\sigma_n}+\frac{t^2}{2\sigma_n^2}+O(\frac{t^3}{\sigma_n^3})$ and noting that $\sigma_n=\theta(n^{1/2})$), we get
\begin{equation}
\alpha_1(e^{t/\sigma_n})\ = \ \alpha_1(1)+\alpha'(1)(\frac{t}{\sigma_n}+\frac{t^2}{2\sigma_n^2}+O(n^{-3/2}))+\frac{\alpha_1''(1)}{2} \left[\frac{t^2}{\sigma^2_n}+O(n^{-3/2})\right]+O(n^{-3/2}).
\end{equation}
Taking the logarithm and using the Taylor expansion $\log(1+x)=x-x^2/2+O(x^3)$ gives us:
\begin{align}\label{log_alpha}
\log \frac{\alpha_1(e^{t/\sigma_n})}{\alpha_1(1)} & \ = \ \log \left( 1+\frac{\alpha_1'(1)}{\alpha_1(1)}\frac{t}{\sigma_n}+\frac{\alpha_1'(1)+\alpha''_1(1)}{\alpha_1(1)}\frac{t^2}{2\sigma^2_n}+O(n^{-3/2}\right)\nonumber\\
& \ = \ \frac{\alpha_1'(1)}{\alpha_1(1)}\frac{t}{\sigma_n}+\frac{\alpha_1'(1)+\alpha''_1(1)}{\alpha_1(1)}\frac{t^2}{2\sigma^2_n}-\left(\frac{\alpha_1'(1)}{\alpha_1(1)}\right)^2\frac{t^2}{2\sigma_n^2}+O(n^{-3/2}).
\end{align}
Substituting \eqref{log_q} and \eqref{log_alpha} into \eqref{log_M_Yn}:
\begin{align}
\log M_{Y_n}(t) & \ = \ -\frac{t\mu_n}{\sigma_n}+n\left(\frac{\alpha_1'(1)}{\alpha_1(1)}\frac{t}{\sigma_n}+\frac{\alpha_1'(1)+\alpha''_1(1)}{\alpha_1(1)}\frac{t^2}{2\sigma^2_n}-\left(\frac{\alpha_1'(1)}{\alpha_1(1)}\right)^2\frac{t^2}{2\sigma_n^2}+O(n^{-3/2}) \right)+o(1)\nonumber\\
& \ = \ \left(n\frac{\alpha'_1(1)}{\alpha_1(1)} - \mu_n\right)\frac{t}{\sigma_n} + n\frac{\alpha_1(1)[\alpha_1'(1)+\alpha_1''(1)] - \alpha_1'(1)^2}{\alpha_1(1)^2}\frac{t^2}{2\sigma_n^2}+o(1).
\end{align}
Using the same notations $A,B,C,D$ as in Theorem \ref{thm_generalGaussian}:
\begin{align}
\log M_{Y_n}(t) & \ = \ \frac{An-\mu_n}{\sigma_n}\cdot t+\frac{Cn}{\sigma_n^2}\cdot \frac{t^2}{2}+o(1)\nonumber\\
& \ = \ \frac{B+o(1)}{\sqrt{Cn+D+o(1)}}\cdot t+\frac{Cn}{Cn+D+o(1)}\cdot \frac{t^2}{2}+o(1)\nonumber\\
& \ = \ \frac{t^2}{2}+o(1).
\end{align}
This implies the moment generating function of $Y_n$ converges to that of the standard normal distribution. So as $n\to\infty$, the moment generating function of  $X_n$ converges to a Gaussian, which implies convergence in distribution.
\hfill $\Box$

\section{Distribution of Gaps}\label{sec:distrgaps}

\subsection{Notation and Counting Lemmas}

In this section we prove our results about gaps between summands arising from $k$-Skipponacci far-difference representations. Specifically, we are interested in the probability of finding a gap of size $j$ among all gaps in the decompositions of integers $x \in [R_k(n),R_k(n+1)]$. In this section, we adopt the notation used in \cite{BBGILMT}. If $\epsilon_i \in \{-1, 1\}$ and
\be x \ = \ \epsilon_j S^{(k)}_{i_j} + \epsilon_{j-1} S^{(k)}_{i_{j-1}} + \cdots + \epsilon_1 S^{(k)}_{i_1} \ee
is a legal far-difference representation (which implies that $i_j = n$), then the gaps are
\be i_j - i_{j-1}, \ \ \ \ i_{j-1} - i_{j-2}, \ \ \ \ \dots,  \ \ \ \ i_2 - i_1. \ee

Note that we do not consider the `gap' from the beginning up to $i_1$, though if we wished to include it there would be no change in the limit of the gap distributions. Thus in any $k$-Skipponacci far-difference representations, there is one fewer gap than summands. The greatest difficulty in the subject is avoiding double counting of gaps, which motivates the following definition.

\begin{defn}[Analogous to Definition 1.4 in \cite{BBGILMT}] \label{GapNotation} \
\begin{itemize}
\item Let $X_{i,i+j}(n)$ denote the number of  integers $x \in [R_k(n),R_k(n+1)]$ that have a gap of length $j$ that starts at $S^{(k)}_i$ and ends at $S^{(k)}_{i+j}$.
\item Let $Y(n)$ be the total number of gaps in the far-difference decomposition for \\ $x \in [R_k(n), R_k(n+1)]$:
\begin{equation} \label{Y(n)}
Y(n) \ := \  \sum_{i=1}^n \sum_{j=0}^n X_{i,i+j}(n).
\end{equation}
Notice that $Y(n)$ is equivalent to the total number of summands in all decompositions for all $x$ in the given interval \emph{minus} the number of integers in that interval. The main term is thus the total number of summands, which is
\be \left[A_{1,1}n + B_{1,1} + o(1)\right] \cdot [R_k(n+1)-R_k(n)] \ = \ A_{1,1} n [R_k(n+1)-R_k(n)], \ee
 as we know from \S\ref{sec:subsecgaussianity} that $\mathbb{E}[\mathcal{K}_n+\mathcal{L}_n]=A_{1,1}n + B_{1,1} + o(1)$.
\item Let $P_n(j)$ denote the proportion of gaps from decompositions of $x$ $\in$ $[R_k(n)$, $R_k(n+1)]$ that are of length $j$:
\begin{equation} \label{P_n(j)}
P_n(j) \ := \  \frac{\sum_{i=1}^{n-j} X_{i,i+j}(n)}{Y(n)},
\end{equation}
and let
\begin{equation} \label{P(j)}
P(j) \ := \  \lim_{n\to\infty} P_n(j)
\end{equation} (we will prove this limit exists).
\end{itemize}
\end{defn}

Our proof of Theorem \ref{thm:gapresult} starts by counting the number of gaps of constant size in the $k$-Skipponacci far-difference representations of integers. To accomplish this, it is useful to adopt the following notation.

\begin{defi} Notation for counting integers with particular $k$-Skipponacci summands. \label{defi:N(S)notation}
\begin{itemize}

\item Let $N(\pm S^{(k)}_i,\pm S^{(k)}_j)$ denote the number of integers whose decomposition begins with $\pm S^{(k)}_i$ and ends with $\pm S^{(k)}_j$.

\item Let $N(\pm F_i)$ be the number of integers whose decomposition ends with $\pm F_i$.

\end{itemize}
\end{defi}

The following results, which are easily derived using the counting notation in Definition \ref{defi:N(S)notation}, are also useful.

\begin{lem} \label{lem:counting}
\begin{equation} \label{N(S^{(k)}_n)shift}
N(\pm S^{(k)}_i,\pm S^{(k)}_j) \ = \  N(\pm S^{(k)}_1,\pm S^{(k)}_{j-i+1}).
\end{equation}
\begin{equation} \label{N(S^{(k)}_n)in-exclusion}
N(-S^{(k)}_1, +S^{(k)}_j) + N(+S^{(k)}_1, +S^{(k)}_j) \ = \  N(+S^{(k)}_j) - N(+S^{(k)}_{j-1}).
\end{equation}
\begin{equation} \label{N(S^{(k)}_n)cardinality}
N(+S^{(k)}_i) \ = \  R_k(i) - R_k(i-1).
\end{equation}
\end{lem}

\begin{proof} First, note that \eqref{N(S^{(k)}_n)shift} describes a shift of indices, which doesn't change the number of possible decompositions. For \eqref{N(S^{(k)}_n)in-exclusion}, we can apply inclusion-exclusion to get
\bea
& & N(-S^{(k)}_1, +S^{(k)}_j) + N(+S^{(k)}_1, +S^{(k)}_j)
\nonumber\\ & & \ \ \ \ \ = \  N(+S^{(k)}_j) - \left[N(+S^{(k)}_2, +S^{(k)}_j) + N(+S^{(k)}_3, +S^{(k)}_j) + \cdots\right] \nonumber\\ & & \ \ \ \ \ = \  N(+S^{(k)}_j) - \left[N(+S^{(k)}_1, +S^{(k)}_{j-1}) + N(+S^{(k)}_2, +S^{(k)}_{j-1}) + \cdots\right] \nonumber\\ & & \ \ \ \ \ = \  N(+S^{(k)}_j) - N(+S^{(k)}_{j-1}).
\eea

Finally, for \eqref{N(S^{(k)}_n)cardinality}, recall that the $k$-Skipponaccis partition the integers into intervals of the form $[S^{(k)}_n-R_k(n-k-2), R_k(n)]$, where $S^{(k)}_n$ is the main term of all of the integers in this range. Thus $N(+F_i)$ is the size of this interval, which is just $R_k(i) - R_k(i-1)$, as desired. \end{proof}

\subsection{Proof of Theorem \ref{thm:gapresult}}

We take a combinatorial approach to proving Theorem \ref{thm:gapresult}. We derive expressions for $X_{i,i+c}(n)$ and $X_{i,i+j}(n)$ by counting, and then we use the Generalized Binet's Formula for the $k$-Skipponaccis in Lemma \ref{Binet-Skipponacci} to reach the desired expressions for $P_n(j)$, and then take the limit as $n\to\infty$.

\begin{proof}[Proof of Theorem \ref{thm:gapresult}] We first consider gaps of length $j$ for $k+2 \le j < 2k+2$, then show that the case with gaps of length $j \ge 2k+2$ follows from a similar calculation. It is important to separate these two intervals as there are sign interactions that must be accounted for in the former that do not affect our computation in the latter. From Theorem \ref{Thm:Far-Diff}, we know that there are no gaps of length $k+1$ or smaller. Using Lemma \ref{lem:counting}, we find a nice formula for $X_{i,i+j}(n)$. For convenience of notation, we will let $R_k$ denote $R_k(n)$ in the following equations:
\begin{align} \label{X(i,i+c)}
X_{i,i+j}(n)
\;\ = \ &\; N(+S^{(k)}_i)\left[N(+S^{(k)}_{n-i-j+1}) - N(+S^{(k)}_{n-i-j})\right] \nonumber \\
\;\ = \ &\; (R_i - R_{i-1})\left[(R_{n-i-j+1} - R_{n-i-j}) - (R_{n-i-j} - R_{n-i-j-1})\right] \nonumber \\
\;\ = \ &\; R_{i-k-1} \cdot (R_{n-i-j-k} - R_{n-i-j-k-1}) \nonumber \\
\;\ = \ &\; R_{i-k-1} \cdot R_{n-i-j-2k-1}.
\end{align}

To continue, we need a tractable expression for $R_k(n)$. Using the results from the Generalized Binet's Formula in Lemma \ref{Binet-Skipponacci}, we can express $R_k(n)$ as
\begin{align} \label{R_nBinet}
R_k(n)
\;\ = \ &\; S^{(k)}_n + S^{(k)}_{n-2k-2} + S^{(k)}_{n-4k-4} + S^{(k)}_{n-6k-6} + \cdots \nonumber \\
\;\ = \ &\; a_1\lambda_1^n + a_1\lambda_1^{n-2k-2} + a_1\lambda_1^{n-4k-4} + a_1\lambda_1^{n-6k-6} + \cdots \nonumber \\
\;\ = \ &\; a_1\lambda_1^n\left[1 + \lambda_1^{-2k-2} + \lambda_1^{-4k-4} + \lambda_1^{-6k-6} + \cdots\right] \nonumber \\
\;\ = \ &\; a_1\lambda_1^n\left[1 + \left(\lambda_1^{-2k-2}\right) + \left(\lambda_1^{-2k-2}\right)^2 + \left(\lambda_1^{-2k-2}\right)^3 + \cdots\right] \nonumber \\
\;\ = \ &\; \frac{a_1\lambda_1^n}{1-\lambda_1^{-2k-2}} + O_k(1)
\end{align} (where the $O_k(1)$ error depends on $k$ and arises from extending the finite geometric series to infinity). We substitute this expression for $R_k(n)$ into the formula from \eqref{X(i,i+c)} for $X_{i,i+j}(n)$, and find
\bea \label{X(i,i+c)Binet}
X_{i,i+j}(n)
& \ = \ & R_{i-k-1} \cdot R_{n-i-j-2k-1} \nonumber\\ & = & \frac{a_1\lambda_1^{i-k-1}(1 + O_k(1))}{1-\lambda_1^{-2k-2}} \cdot \frac{a_1\lambda_1^{n-i-j-2k-1}(1 + O_k(1))}{1-\lambda_1^{-2k-2}} \nonumber\\
& \ = \ & \frac{a_1^2\lambda_1^{n-j-3k-2}(1 + O_k(\lambda_1^{-i} + \lambda_1^{-n+i+j})}{\left(1-\lambda_1^{-2k-2}\right)^2}.
\eea
We then sum $X_{i,i+j}(n)$ over $i$. Note that almost all $i$ satisfy $\log\log n \ll i \ll n - \log \log n$, which means the error terms above are of significantly lower order (we have to be careful, as if $i$ or $n-i$ is of order 1 then the error is of the same size as the main term). Using our expression for $Y(n)$ from Definition \ref{GapNotation} we find
\begin{align} \label{P_n(c)proof}
P_n(j)
\;\ = \ &\; \frac{\sum_{i=1}^{n-j} X_{i,i+j}(n)}{Y(n)} \nonumber \\
\;\ = \ &\; \frac{a_1^2\lambda_1^{n-j-3k-2}(n-j)(1 + o_k(n \lambda_1^n))}{\left[A_{1,1}n+B_{1,1} + o(1)\right] \cdot \left(1-\lambda_1^{-2k-2}\right)^2 \cdot a_1\lambda_1^n(\lambda_1-1) + O(\lambda_1^n)}.
\end{align}
Taking the limit as $n\to\infty$ yields
\begin{align} \label{P(c)proof}
P(j) \ = \ \lim_{n\to\infty} P_n(j) \ = \ \frac{a_1\lambda_1^{-3k-2}}{A_{1,1} \left(1-\lambda_1^{-2k-2}\right)^2 (\lambda_1-1)}\lambda_1^{-j}.
\end{align}

For the case where $j \ge 2k+2$, the calculation is even easier, as we no longer have to worry about sign interactions across the gap (that is, $S^{(k)}_i$ and $S^{(k)}_{i+j}$ no longer have to be of opposite sign). Thus the calculation of $X_{i,i+j}(n)$ reduces to
\begin{align} \label{X(i,i+j)}
X_{i,i+j}(n)
\;\ = \ &\; N(+S^{(k)}_i)N(+S^{(k)}_{n-i-j})\nonumber \\
\;\ = \ &\; (R_i - R_{i-1})(R_{n-i-j} - R_{n-i-j-1}) \nonumber \\
\;\ = \ &\; R_{i-k-1} \cdot R_{n-i-j-k-1}.
\end{align}
We again use \eqref{R_nBinet} to get
\begin{equation} 
X_{i,i+c}(n)
\;\ = \ \; R_{i-k-1} \cdot R_{n-i-j-k-1}
\;\ = \ \; \frac{a_1^2\lambda_1^{n-j-2k-2}(1 + o_k(\lambda_1^n))}{\left(1-\lambda_1^{-2k-2}\right)^2}.
\end{equation}
Which, by a similar argument as before, gives us
\begin{equation} \label{P(j)proof}
P(j)
\;\ = \ \; \frac{a_1\lambda_1^{-2k-2}}{A_{1,1} \left(1-\lambda_1^{-2k-2}\right)^2 (\lambda_1-1)}\lambda_1^{-j},
\end{equation} completing the proof.\end{proof}

\section{Generalized Far-Difference Sequences}\label{sec:genfardiffseq}

The $k$-Skipponaccis give rise to unique far-difference representations where same signed indices are at least $k + 2$ apart and opposite signed indices are at least $2k+2$ apart. We consider the reverse problem, namely, given a pair $(s,d)$ of positive integers, when does there exist a sequence $\{a_n\}$ such that every integer has a unique far-difference representation where same signed indices are at least $s$ apart and opposite signed indices are at least $d$ apart. We call such representations $(s,d)$ far-difference representations.

\subsection{Existence of Sequences}

\begin{proof}[Proof of Theorem \ref{farDiffRec}]
Define
\be
R^{(s,d)}_n \ = \ \sum_{i=0}^{\lfloor n/s\rfloor}a_{n-is}\ = \  a_n+a_{n-s}+a_{n-2s}+\cdots.
\ee
For each $n$, the largest number that can be decomposed using $a_n$ as the largest summand is $R^{(s,d)}_n$, while the smallest one is $a_n-R^{(s,d)}_{n-d}$. It is therefore natural to break our analysis up into intervals $I_n=[a_n-R^{(s,d)}_{n-d},R^{(s,d)}_n]$.

We first prove by induction that
\begin{equation}\label{condition1}
a_n\ = \ R^{(s,d)}_{n-1}+R^{(s,d)}_{n-d}+1,
\end{equation}  or equivalently, $a_n-R^{(s,d)}_{n-d}=R^{(s,d)}_{n-1}+1$ for all $n$, so that these intervals $\{I_n\}_{n=1}^\infty$ are disjoint and cover $\mathbb{Z}^+$.

Indeed, direct calculation proves \eqref{condition1} is true for $n=1,\dots,\max(s,d)$. For $n>\max(s,d)$, assume it is true for all positive integers up to $n-1$. We have
\begin{align}
a_{n-s}
\ = \ R^{(s,d)}_{n-s-1}+R^{(s,d)}_{n-s-d}+1
\ =& \ (R^{(s,d)}_{n-1}-a_{n-1})+(R^{(s,d)}_{n-d}-a_{n-d})+1 \nonumber \\
\Rightarrow \ R^{(s,d)}_{n-1}+R^{(s,d)}_{n-d}+1
\ =& \  a_{n-s}+a_{n-1}+a_{n-d}\ = \ a_n.
\end{align}
This implies that \eqref{condition1} is true for $n$ and thus true for all positive integers.\\

We prove that every integer is uniquely represented as a sum of $\pm a_n$'s in which every two terms of the same sign are at least $s$ apart in index and every two terms of opposite sign are at least $d$ apart in index. We prove by induction that any number in the interval $I_n$ has a unique $(s,d)$ far-difference representation with main term (the largest term) be $a_n$.

It is easy to check for $n\leq \max(s,d)$. For $n>\max(s,d)$, assume it is true up to $n-1$. Let $x$ be a number in $I_n$, where $a_n-R^{(s,d)}_{n-d}\leq x\leq R^{(s,d)}_n$. There are two cases to consider.
\begin{enumerate}
\item If $a_n\leq x\leq R^{(s,d)}_n$, then either $x=a_n$ or $1\leq x-a_n\leq R^{(s,d)}_n-a_n=R^{(s,d)}_{n-s}$. By the induction assumption, we know that $x-a_n$ has a far-difference representation with main term of at most $a_{n-s}$. It follows that $x=a_n+(x-a_n)$ has a legal decomposition.
\item If $a_n-R^{(s,d)}_{n-d}\leq x<a_n$ then $1\leq a_n-x\leq R^{(s,d)}_{n-d}$. By the induction assumption, we know that $a_n-x$ has a far-difference representation with main term at most $a_{n-d}$. It follows that $x=a_n-(a_n-x)$ has a legal decomposition.
\end{enumerate}

To prove uniqueness, assume that $x$ has two difference decompositions $\sum_i \pm a_{n_i}=\sum_i \pm a_{m_i}$, where $n_1>n_2>\dots$ and $m_1>m_2>\dots$. Then it must be the case that $x$ belongs to both $I_{n_1}$ and $I_{m_1}$. However, these intervals are disjoint, so by contradiction we have $n_1=m_1$. Uniqueness follows by induction.
\end{proof}

\begin{remark}
As the recurrence relation of $a_n$ is symmetric between $s$ and $d$, it is the initial terms that define whether a sequence has an $(s,d)$ or a $(d,s)$ far-difference representation.
\end{remark}

\begin{cor}
The Fibonacci numbers $\{1,2,3,5,8,\dots\}$ have a $(4,3)$ far-difference representation.
\end{cor}

\begin{proof}
We can rewrite Fibonacci sequence as $F_1=1, F_2=2, F_3=3$, $F_4=F_3+F_1+1$, and $F_n=F_{n-1}+F_{n-2} = F_{n-1} + (F_{n-3}+F_{n-4})$ for $n\geq 5$.
\end{proof}

\begin{cor}
The $k$-Skipponacci numbers, which are defined as $a_n=n$ for $n\leq k$ and $a_{n+1}=a_n+a_{n-k}$ for $n>k$, have a $(2k+2,k+2)$ far-difference representation.
\end{cor}

\begin{proof}
This follows from writing the recurrence relation as $a_n=a_{n-1}+a_{n-k-1}=a_{n-1}+a_{n-k-2}+a_{n-2k-2}$ and using the same initial conditions.
\end{proof}

\begin{cor}
Every positive integer can be represented uniquely as a sum of $\pm 3^n$ for $n=0,1,2,\dots$.
\end{cor}

\begin{proof}
The sequence $a_n=3^{n-1} $ satisfies $a_n=3a_{n-1}$, which by our theorem has an $(1,1)$ far-difference representation.
\end{proof}

\begin{cor}
Every positive integer can be represented uniquely as $\sum_i \pm 2^{n_i}$ where $n_1>n_2>\dots$ and $n_i\geq n_{i-1}+2$, so any two terms are apart by at least two.
\end{cor}
\begin{proof}
The sequence $a_n=2^n $ satisfies $a_n=a_{n-1}+2a_{n-2}$, which by our theorem has a $(2,2)$ far-difference representation.
\end{proof}

\subsection{Non-uniqueness}

We consider the inverse direction of Theorem \ref{farDiffRec}. Given positive integers $s$ and $d$, how many increasing sequences are there that have $(s,d)$ far-difference representation?

The following argument suggests that any sequence $a_n$ that has $(s,d)$ far-difference representation should satisfy the recurrence relation $a_n=a_{n-1}+a_{n-s}+a_{n-d}$. If we want the intervals $[a_n-R_{n-d},R_n]$ to be disjoint, which is essential for the unique representation, we must have
 \begin{equation}
 a_n-R_{n-d}\ = \ R_{n-1}+1.
 \end{equation}
Replacing $n$ by $n-s$ gives us
\begin{equation}
a_{n-s}-R_{n-d-s}\ = \ R_{n-1-s}+1.
\end{equation}
When we subtract those two equations and note that $R_k-R_{k-s}=a_k$, we get
\begin{equation}
a_n-a_{n-s}-a_{n-d}\ = \ a_{n-1}
\end{equation}
or $a_n=a_{n-1}+a_{n-s}+a_{n-d}$, as desired. What complicates this problem is the choice of initial terms for this sequence. Ideally, we want to choose the starting terms so that we can guarantee that every integer will have a unique far-difference representation. We have shown this to be the case which for the initial terms defined in Theorem \ref{farDiffRec}, which we refer as the \emph{standard} $(s,d)$ sequence. However, it is not always the case that the initial terms must follow the standard model to have a unique far-difference representation. In fact, it is not even necessary that the sequence starts with $1$.

In other types of decompositions where only positive terms are allowed, it is often obvious that a unique increasing sequence with initial terms starting at $1$ is the desired sequence. However, in far-difference representations where negative terms are allowed, it may happen that a small number (such as 1) arises through subtraction of terms that appear later in the sequence. Indeed, if $(s,d)=(1,1)$, we find several examples where the sequence need not start with 1.

\begin{exa}\label{exa:one}
The following sequences have a $(1,1)$ far-difference representation.
\begin{itemize}

\item $a_1=2,a_2=6$ and $a_n=3^{n-1}$ for $n\geq 3$
\item $a_1=3,a_2=4$ and $a_n=3^{n-1}$ for $n\geq 3$
\item $a_1=1,a_2=9,a_3=12$ and $a_n=3^{n-1}$ for $n\geq 4$
\end{itemize}
\end{exa}

\begin{exa}\label{exa:two} For each positive integer $k$, the sequence $B_{k}$, defined by $B_{k,i}= \pm 2 \cdot 3^{i-1}$ for $i=k+1$ and $B_{k,i}= \pm 3^{i-1}$ otherwise,  has a $(1,1)$ far-difference representation. \end{exa}

\noindent
We prove this by showing that there is a bijection between a decomposition using the standard sequence $b_n=\pm 3^{n-1}$ and a decomposition using $B_{k}$. First we give an example: For $k=2$, the sequence is $1,3,2\cdot 3^2,3^3,3^4,\dots$
\begin{align*}
763 &\ = \  1-3+3^2+3^3+3^6\nonumber\\
&\ = \ 1-3+(3^3-2\cdot 3^2)+3^3+3^6\nonumber\\
&\ = \ 1-3-2\cdot 3^2+2\cdot 3^3+3^6\nonumber\\
&\ = \ 1-3-2\cdot 3^2+3^4-3^3+3^6\nonumber\\
&\ = \ B_{2,0}-B_{2,1}-B_{2,2} - B_{2,3} + B_{2,4} + B_{2,6}.
\end{align*}
\noindent Conversely,
\begin{align*}
763 &\ = \ B_{2,0}-B_{2,1}-B_{2,2} - B_{2,3} + B_{2,4} + B_{2,6} \nonumber\\
&\ = \ 1-3-2.3^2-3^3+3^4+3^6\nonumber\\
&\ = \ 1-3- (3^3-3^2)-3^3+3^4+3^6\nonumber\\
&\ = \ 1-3+3^2-2.3^3+3^4+3^6\nonumber\\
&\ = \ 1-3+3^2-(3^4-3^3)+3^4+3^6\nonumber\\
&\ = \ 1-3+3^2+3^3+3^6.
\end{align*}

\noindent To prove the first direction, assume $x=\sum_{i\in I} 3^i-\sum_{j\in J}3^j$ where $I,J$ are disjoint subsets of $\mathbb{Z}^+$. If $k$ is not in $I\cup J$, this representation is automatically a representation of $x$ using $B_{k}$. Otherwise, assume $k\in I$, we replace the term $3^k$ by $3^{k+1}-2 \cdot 3^k=B_{k,k+2}-B_{k,k+1}$. If $k+1\notin I$, again $x$ has a $(1,1)$ far-difference representation of $B_{k}$. Otherwise, $x$ has the term $2 \cdot 3^{k+1}$ in its representation, we can replace this term by $3^{k+2}-3^{k+1}$. Continue this process, stopping  if $k+2\notin I$ and replacing the extra term if $k+2\in I$. Hence we can always decompose $x$ by $\pm B_{k,i}$.

Conversely, suppose $x=\sum_{i\in I} B_{k,i}-\sum_{j\in J} B_{k,j}$. If $k+1\notin I\cup J$, this representation is automatically a representation of $x$ using $\pm 3^n$. If not, assume $k+1\in I$, we replace $B_{k,k+1}=2\cdot 3^k$ by $3^{k+1}-3^k$. If $k+2\notin I$ we are done, if not, $x$ has a term $2\cdot 3^{k+1}$, replace this one by $3^{k+2}-3^{k+1}$ and continue doing this, we always get a decomposition using $\pm 3^n$. Since there is only one such decomposition, the decomposition using $\pm B_{k,i}$ must also be unique. \hfill $\Box$

\begin{remark}
From Example \ref{exa:two}, we know that there is at least one infinite family of sequences that have $(1,1)$ far-difference representations. Example \ref{exa:one} suggests that there are many other sequences with that property and, in all examples we have found to date, there exists a number $k$ such that the recurrence relation $a_n=3a_{n-1}$ holds for all $n\geq k$.
\end{remark}

\section{Conclusions and Further Research}

In this paper we extend the results of \cite{Al, MW1, BBGILMT} on the Fibonacci sequence to all $k$-Skipponacci sequences. Furthermore, we prove there exists a sequence that has an $(s,d)$ far-difference representation for any positive integer pair $(s,d)$. This new sequence definition further generalizes the idea of far-difference representations by uniquely focusing on the index restrictions that allow for unique decompositions. Still many open questions remain that we would like to investigate in the future. A few that we believe to be the most important and interesting include:
\begin{itemize}
\item[(1)]
Can we characterize all sequences that have $(1,1)$ far-difference representations? Does every such sequence converge to the recurrence $a_n=3a_{n-1}$ after first few terms?

\item[(2)] For $(s,d)\neq (1,1)$, are there any \emph{non-standard} increasing sequences that have a $(s,d)$ far-difference representation? If there is such a sequence, does it satisfy the recurrence relation stated in Theorem \ref{farDiffRec} after the first few terms?

\item[(3)] Will the results for Gaussianity in the number of summands still hold for any sequence that has an $(s,d)$ far-difference representation?

\item[(4)] How are the gaps in a general $(s,d)$ far-difference representation distributed?

\end{itemize}

\appendix
\section{Proof of Lemma \ref{Lem:R+R=S-1}}\label{sec:proofsfromsecfardiffreplemmas}

\begin{proof}[Proof of Lemma \ref{Lem:R+R=S-1}]
We proceed by induction on $n$. It is easy to check that \eqref{lemma1} holds for $n=1,\dots,2k+2$. For $n\geq 2k+2$, assume  that the relationship in \eqref{lemma1} holds for all integers up to to $n-1$. We claim that it further holds up to $n$. We see that
\begin{align}
R_k(n-1) + R_k(n-k-2)
\;\ = \ &\; S^{(k)}_{n-1} + S^{(k)}_{n-k-2} + S^{(k)}_{n-2k-3} + S^{(k)}_{n-3k-4} + . . . \nonumber  \\
\;\ = \ &\; S^{(k)}_{n-1} + \left(S^{(k)}_{n-k-2} + S^{(k)}_{n-2k-3} + S^{(k)}_{n-3k-4} + . . .\right) \nonumber  \\
\;\ = \ &\; S^{(k)}_{n-1} + \left[ R_k(n-k-2) + R_k(n-2k-3) \right] \nonumber \\
\;\ = \ &\; S^{(k)}_{n-1} + S^{(k)}_{n-k-1} - 1 \nonumber \\
\;\ = \ &\; S^{(k)}_{n} - 1,
\end{align} completing the proof. \end{proof}

\section{Proof of Proposition \ref{prop:mainres}}\label{sec:propmainres}

\noindent Before we prove Proposition \ref{prop:mainres}, we define a few helpful equations. Let $A_w(z)$ and $\hat{A}_w(z)$ denote the denominators of the generating functions in \eqref{genfn} and \eqref{Genfn2}, respectively. Making the substitution $(x,y) = (w^a,w^b)$ in each case gives us the following expressions:
\begin{equation} \label{Aw(z)}
A_w(z) \ = \  1-2z+z^2-(w^a+w^b)\left(z^{2k+2}-z^{2k+3}\right)-w^{a+b}\left(z^{2k+4}-z^{4k+4}\right)
\end{equation} and
\begin{equation} \label{hat(A)w(z)}
\hat{A}_w(z) \ = \  1-z-(w^a+w^b)z^{2k+2}-w^{a+b}\sum_{j=2k+4}^{4k+3}z^j.
\end{equation}
Notice that the coefficients of $A(z)$ are polynomials in one variable, and therefore continuous. This implies that the roots of $A(z)$ are continuous as well.
Since we are interested only in the region near the point $w=1$, it is enough to prove the results of part $(a)$ and $(b)$ at $w=1$. Thus we use the following expressions as well, which are formed by substituting $w=1$ into \eqref{Aw(z)} and \eqref{hat(A)w(z)}, respectively:
\begin{equation} \label{A(z)}
A(z) \ = \  1-2z+z^2-2z^{2k+2}+2z^{2k+3}-z^{2k+4}+z^{4k+4}
\end{equation} and
\begin{equation} \label{hat(A)(z)}
\hat{A}(z) \ = \  1-z-2z^{2k+2}-\sum_{j=2k+4}^{4k+3}z^j.
\end{equation}

\noindent \textbf{Proof of Proposition \ref{prop:mainres}(a).} It is enough to prove $A(z)$ from \eqref{A(z)} does not have multiple roots. We begin by factoring $A(z)$:
\begin{align}
A(z)
\;\ = \ &\; 1-2z+z^2-2z^{2k+2}+2z^{2k+3}-z^{2k+4}+z^{4k+4} \nonumber \\
\;\ = \ &\;(1-2z^{2k+2}+z^{4k+4})+(2z^{2k+3}-2z)+(z^2-z^{2k+4}) \nonumber \\
\;\ = \ &\;(z^{2k+2}-1)^2+2z(z^{2k+2}-1)-z^2(z^{2k+2}-1) \nonumber \\
\;\ = \ &\;(z^{2k+2}-1)(z^{2k+2}-1+2z-z^2) \nonumber \\
\;\ = \ &\;(z^{2k+2}-1)(z^{2k+2}-(z-1)^2) \nonumber \\
\;\ = \ &\;(z^{2k+2}-1)(z^{k+1}+z-1)((z^{k+1}-z+1).
\end{align}

Let $a(z)=z^{2k+2}-1,b(z)=z^{k+1}+z-1,$ and $ c(z)=z^{k+1}-z+1$. We begin by proving that $a(z),b(z)$ and $c(z)$ are pairwise co-prime. Here we use the fact that $\gcd(p,q)=\gcd(p,q-rp)$ for any polynomials $p,q,r\in\mathbb{Z}[x]$. We have
\bea \gcd(a,b,c) & \ = \ & \gcd \left(z^{2k+2}-1,z^{2k+2}-z^2+2z-1\right)\nonumber\\ & \ = \ & \gcd\left(z^{2k+2}-1, z^2-2z\right)\ = \ 1.
\eea The last equality holds because $z^2-2z=z(z-2)$ has only two roots $z=0,2$ neither of which are a root of $z^{2k+2}-1$. It follows that $\gcd(a,b) = \gcd(a,c) = 1$. Similarly, $\gcd(b,c)$ $=$ $\gcd(z^{k+1}+z-1$, $z^{k+1}-z+1)$ $=$ $\gcd\left(z^{k+1}+z-1,2z-2\right)$ $=$ $1$ because $2z-2$ has only one root $z=1$ which is not a root of $z^{k+1}+z-1$. It follows that $\gcd(b,c)=1$ as well.\\
\\
We prove that the polynomials $a(z), b(z)$ and $c(z)$ do not have repeated roots. The roots of $a(z) = z^{2k+2}-1$ are $e^{i\pi \ell/(2k+2)}$ for $\ell=1,\dots,2k+2$ and are therefore distinct. For $b(z)$ and $c(z)$, we prove that $\gcd\left(b(z),b'(z)\right)$ $=$ $\gcd\left(c(z),c'(z)\right)$ $=$ $1$, and therefore that they do not have repeated roots either. Indeed, we have
\begin{align}
\gcd[b(z),b'(z)]
\;\ = \ &\; \gcd \left(z^{k+1}+z-1,(k+1)z^k+1\right) \nonumber \\
\;\ = \ &\; \gcd \left(z^{k+1}+z-1, z^{k+1}+z-1-\frac{z}{k+1}\left[(k+1)z^k+1\right]\right) \nonumber \\
\;\ = \ &\; \gcd \left(z^{k+1}+z-1,\frac{k}{k+1}z-1 \right) \;\ = \ \; 1,
\end{align} where the last equality again holds since $\frac{k}{k+1}z-1$ has only one root $z=\frac{k+1}{k}$ which is not a root of $z^{k+1}+z-1$. By a similar method, we can prove that $\gcd[c(z),c'(z)]=1$. It follows that $A(z)=a(z)b(z)c(z)$ has no repeated roots. \qed \\

\noindent \textbf{Proof of Proposition \ref{prop:mainres}(b).}  We need to prove that $\hat{A}(z)$ $=$ $1-z-2z^{2k+2} -\sum_{j=2k+4}^{4k+3} z^j$ has only one real root $e_1$ on the interval $(0,\infty)$, $|e_1|<1$, and all other roots $e_i$ with $i\geq 2$ satisfy $|e_1|/ |e_i|<\lambda<1$ for some $\lambda$.\\
\\
 Indeed, first note that $\hat{A}(0)>0$ while $\hat{A}(1)<0$, thus $\hat{A}(x)$ must have at least one root $e_1$ on $(0,1)$. Moreover, since $\hat{A}'(z)$ $=$ $-1-2(2k+2)z^{2k+1}$ $-$ $\sum_{j=2k+4}^{4k+3}jz^{j-1}$, which is negative whenever $z\geq 0$, the function is strictly decreasing on $(0,\infty)$. It follows that $e_1$ is the only real root of $\hat{A}$ in this interval. Let $e_i$ be another root of $\hat{A}(z)$, and assume that $|e_i|\leq e_1$. Then $|e_i|^j\leq |e_1|^j$ $=e_1^j$ for any $j\in\mathbb{Z}^{+}$. Rearranging $\hat{A}(e_i)=0$ to be $1=e_i+2e_i^{2k+2}$ $+$ $\sum_{j=2k+4}^{4k+3}e_i^j$ and applying the generalized triangle inequality, we find
\begin{align}
1
\;\ = \ &\; |1| \;\ = \ \; \left| e_i+2e_i^{2k+2}+\sum_{j=2k+4}^{4k+3}e_i^j\right| \;\leq\; |e_i|+2|e_i|^{2k+2}+\sum_{j=2k+4}^{4k+3}|e_i|^j \nonumber \\
\;\leq&\; |e_1|+2|e_1|^{2k+2}+\sum_{j=2k+4}^{4k+3}|e_1|^j \;\ = \ \; e_1+2e_1^{2k+2}+ \sum_{j=2k+4}^{4k+3}e_1^j \;\ = \ \; 1.
\end{align}
Hence the equalities hold everywhere, which implies that $|e_i|=e_1$ and all the complex numbers $e_i, e_i^{2k+2}$ and $e_i^j$ lie on the same line in the complex plane. Since their sum is 1, they must all be real numbers, and since $e_i^{2k+2}>0$,  $e_i$ will be positive. It follows that $e_i = e_1$. However, this is a contradiction because, as we proved before, $e_1$ is a non-repeated root of $\hat{A}(z)$. It follows that $|e_i|>|e_1|$ for any $i\geq 2$.
Let $\lambda=\max_{i\geq 2} \sqrt{e_1/|e_i|}$, then $|e_1|/|e_i|<\lambda<1$ for all $i\geq 2$. \hfill $\Box$

\ \\

\noindent \textbf{Proof of Proposition \ref{prop:mainres}(c).}
Since $\hat{A}(e_1(w))=0$ and the function is continuous, in some small neighborhood $\Delta w$ we have $\hat{A}[e_1(w+\Delta w)]=\epsilon$ for some small $\epsilon$.  This implies
\bea \label{proof3.1c}
\epsilon & \ = \ &\hat{A}[e_1(w)]-\hat{A}[e_1(w+\Delta w)] \nonumber \\
 & \ = \ & \left[1-e_1(w)-(w^a+w^b)e_1(w)^{2k+2}-w^{a+b}\sum_{j = 2k+4}^{4k+3}e_1(w)^j\right] \ \ -\ \left[1-e_1(w+\Delta w)\right. \nonumber \\
& & \left. \ - \ ((w+\Delta w)^a
+(w+\Delta w)^b)e_1(w+\Delta w)^{2k+2} -(w+\Delta w)^{a+b}\sum_{j=2k+4}^{4k+3}e_1(w+\Delta w)^j\right] \nonumber \\
 & \ = \ & e_1(w+\Delta w)-e_1(w)+(w^a+w^b)\Bigg[e_1(w+\Delta w)^{2k+2}-e_1(w)^{2k+2} \nonumber \\
 &  & \ \  +w^{a+b}\sum_{j=2k+4}^{4k+3}\big[e_1(w+\Delta w)^j-e_1(w)^j\big]\Bigg] +\ e_1(w+\Delta w)^{2k+2}\Big[(w+\Delta w)^a-w^a \nonumber \\
& & \ \ +(w+\Delta w)^b-w^b\Big]+\sum_{j=2k+4}^{4k+3} \Big[e_1(w+\Delta w)^j][(w+\Delta w)^{a+b}-w^{a+b}\Big] \nonumber \\
 & \ = \ &\left[e_1(w+\Delta w)-e_1(w)\right] \cdot \left[1+(w^a+w^b)\sum_{i=0}^{2k+1}e_1(w+\Delta w)^ie_1(w)^{2k+1-i}\right. \nonumber \\
& & \ \  +\ \left. w^{a+b}\sum_{j=2k+4}^{4k+3}\sum_{i=0}^{j-1}e_1(w+\Delta w)^ie_1(w)^{j-1-i}\right]  \nonumber \\
& & \ \ + \left[\Delta w\right] \cdot \Bigg[ e_1(w+\Delta w)^{2k+2} \sum_{i=0}^{a-1}(w+\Delta w)^iw^{a-1-i} +  \sum_{i=0}^{b-1}(w+\Delta w)^iw^{b-1-i} \nonumber \\
& & \ \  + \sum_{j=2k+4}^{4k+3}e_1(w+\Delta w)^j \cdot \sum_{i=0}^{a+b-1}(w+\Delta w)^iw^{a+b-1-i}\Bigg].
\eea
Since $e_i(w)$ is continuous, the coefficient of $|e_i(w + \Delta) - e_i(w)|$ converges as $\Delta w \to 0$, and its limit is
\begin{equation} \label{e'denom}
1+(w^a+w^b)(2k+2)e_1(w)^{2k+1}+w^{a+b}\sum_{j=2k+4}^{4k+3}je_1(w)^{j-1}.
\end{equation}
This is the derivative of $-A_w(z)$ at $e_i(w)$, which is non-zero since $A_w(z)$ has no multiple roots. Then, similar to the arguments in \cite{MW1}, since $w^a$, $w^b$ and $w^{a + b}$ are differentiable at $w=1$, the coefficient of $\Delta w$ in \eqref{proof3.1c} converges as well, with limit
\begin{equation} \label{e'num}
\left(aw^{a-1}+bw^{b-1}\right)e_1(w)^{2k+2}+(a+b)w^{a+b-1}\sum_{j=2k+4}^{4k+3}e_1(w)^j.
\end{equation}
Thus we can rearrange the terms in \eqref{proof3.1c}  and take the limit. Note that when $\Delta w\to 0,\epsilon\to 0$ we have
\begin{align}\label{end3.1c}
e_1'(w)
\;\ = \ &\; \lim_{\Delta w\to 0} \frac{e_1(w+\Delta w)-e_1(w)}{\Delta w} \nonumber \\
\;\ = \ &\; -\frac{(aw^{a-1}+bw^{b-1})e_1(w)^{2k+2}+(a+b)w^{a+b-1}\sum_{j=2k+4}^{4k+3}e_1(w)^j}{1+(w^a+w^b)(2k+2)e_1(w)^{2k+1}+w^{a+b}
\sum_{j=2k+4}^{4k+3}je_1(w)^{j-1}},
\end{align} as desired. Further notice that since $e_1(w)$ is a positive real root of our generating function, it is easy to see that the denominator of this derivative is also a positive real number. Since taking further derivatives will utilize the quotient rule (and thus will only include larger powers of the denominator) it is clear that this root is $\ell$ times differentiable for any positive integer $\ell$.

\section{Proof of lemma \ref{lem_variance_grow}}\label{sec:proof_lem_variance_grow}

\noindent \textbf{Proof of \eqref{nonzero_mean}.}
To prove the first equality note that
\begin{equation}
\alpha_1'(w)\ = \ \left(\frac{1}{e_1(w)}\right)'\ = \ -\frac{e'_1(w)}{e^2_1(w)},
\end{equation}
which implies
\begin{equation}\label{alpha_and_e}
\frac{\alpha_1'(w)}{\alpha_1(w)}\ = \ \frac{-e'_1(w)}{e^2_1(w)}\cdot \frac{1}{1/e_1(w)} \ = \ -\frac{e'_1(w)}{e_1(w)}.
\end{equation}
By \eqref{eprime},
\begin{equation}
e'(1)\ = \ \frac{(a+b)e_1(1)^{2k+2}+\sum_{j=2k+4}^{4k+3}e_1(1)^j}{1+2(2k+2)e_1(1)^{2k+1}+\sum_{j=2k+4}^{4k+3} je_1(1)^{j-1}}>0
\end{equation}
because $a+b>0$ and $e_1(1)>0$. Thus $-e'_1(1)/e_1(1)\neq 0$.

\ \\

\textbf{Proof of \eqref{nonzero_variance}.} The first equality follows directly from \eqref{alpha_and_e}.
Let $h_{a,b}(w):= -\frac{we_1'(w)}{e_1(w)}$. We want to prove $h'_{a,b}(1)\neq 0$.
By \eqref{eprime}, we have
\begin{equation} \label{h(w)}
h_{a,b}(w)\ = \ \frac{we_1'}{e_1}\ = \ \frac{(aw^a+bw^b)e_1(w)^{2k+2}+(a+b)w^{a+b}\sum_{j=2k+4}^{4k+3}e_1(w)^j}{e_1+(w^a+w^b)(2k+2)e_1(w)^{2k+2}+w^{a+b}\sum_{j=2k+4}^{4k+3}je_1(w)^j}.
\end{equation}

Let the numerator and denominator of \eqref{h(w)} be $h_1(w)$ and $h_2(w)$, respectively. Further assume that $h'_{a,b}(1)=0$, or equivalently $\left(\frac{h_1(w)}{h_2(w)}\right)'\Big|_{w=1}=0$. Then by the quotient rule, we get $h_1'(1)h_2(1)-h_1(1)h_2'(1)=0$. This implies that
\begin{equation} \label{h1/h2}
\frac{h'_1(1)}{h'_2(1)} \;\ = \ \; \frac{h_1(1)}{h_2(1)} \;\ = \ \; -\frac{e_1'(1)}{e_1(1)} \;\Rightarrow\; 0 \;\ = \ \; h_1'(1)e_1(1) + h_2'(1)e_1'(1).
\end{equation}

For simplicity, we return to our previous notation in which we let $e_1$ represent $e_1(1)$ and $e_1'$ represent $e_1'(1)$. Direct calculation of $h_1'$ and $h_2'$ yields
\begin{align} \label{h'calc}
\;0\; \ = \ & \;(a^2+b^2)e_1^{2k+2}+(a+b)(2k+2)e_1^{2k+1}e_1'+(a+b)^2\sum_{j=2k+4}^{4k+3}e_1^j+(a+b)\sum_{j=2k+4}^{4k+3}je_1^{j-1} \nonumber \\
 & \; + e_1'+(a+b)(2k+2)e_1^{2k+2}+2(2k+2)^2e_1^{2k+1}e_1'+(a+b)\sum_{j=2k+4}^{4k+3}je_1^j+\sum_{j=2k+4}^{4k+3}j^2e_1^{j-1} \nonumber \\
\ = \ & \; e_1^{2k+1}\left[(a^2+b^2)e_1^2+2(a+b)(2k+2)e_1e_1'+2(2k+2)(e_1')^2\right] \nonumber \\
& \; +\sum_{j=2k+4}^{4k+3}e_1^{j-1}\left[(a+b)^2e_1^2+2(a+b)je_1e_1'+j^2(e_1')^2\right]\nonumber \\
\ \geq\ & \; e_1^{2k+1}\frac{1}{2}\left[ (a+b)e_1+2(2k+2)e_1'\right]^2+\sum_{j=2k+4}^{4k+3}e_1^{j-1}\left[(a+b)e_1+je_1'\right]^2
\; \geq \; 0. \end{align}
Here we used the fact that $a^2+b^2\geq \frac{1}{2}(a+b)^2$, and that $x^2 > 0$ for any real $x \neq 0 $. In the last line of \eqref{h'calc}, this implies that $(a+b)e_1+je_1'=0$. We can re-express this as $j\ = \ -\frac{(a+b)e_1}{e_1'}$ (since $e_1'(1)\neq 0$) for every $j\in\{2k+4,\dots,4k+3\}$. This is a contradiction, since $j$ must be an integer, and it follows that $h'_{a,b}(1)\neq 0$. \qed \\




\ \\

\end{document}